\documentclass[12pt,a4paper]{amsart}
\usepackage{blindtext}
\usepackage{amsmath}
\usepackage{graphicx}
\usepackage{amsfonts}
\usepackage{amssymb}
\usepackage{mathtools}

\usepackage{color, soul}
\usepackage{mathrsfs}
\usepackage{epsfig}
\usepackage{url}
\usepackage{mathrsfs,a4wide}
\usepackage{calligra}
\usepackage{bm}
\usepackage{tikz,comment}
\usepackage{hyperref}
\usepackage{comment}
\usepackage{imakeidx}
\usepackage{enumerate}

\newtheorem{theorem}{\bf Theorem}[section]
\newtheorem{proposition}{\bf Proposition}[section]
\newtheorem{lemma}{\bf Lemma}[section]

\newtheorem{remark}{\bf Remark}[section]
\newtheorem{definition}{\bf Definition}[section]
\newcommand{\average}{{\mathchoice {\kern1ex\vcenter{\hrule
height.4pt width 8pt depth0pt}
\kern-11pt} {\kern1ex\vcenter{\hrule height.4pt width 4.3pt
depth0pt} \kern-7pt} {} {} }}
\newcommand{\ave}{\average\int}

\newcommand{\I}{\mathcal{I}}

\newcommand{\ad}{\mathrm{ad}}
\newcommand{\Airy}{\mathsf{A}}

\newcommand{\Dive}{\mathrm{Div}\,}
\newcommand{\de}{\bm{\delta}}

\newcommand{\sym}{\mathrm{sym}}

\newcommand{\ED}{\mathscr{ED}}
\newcommand{\WD}{\mathscr{WD}}

\newcommand{\defects}{\mathsf{D}}

\newcommand{\cof}{\mathrm{cof\,}}
\newcommand{\ep}{\varepsilon}

\newcommand{\C}{\mathbb{C}}
\newcommand{\R}{\mathbb{R}}
\newcommand{\N}{\mathbb{N}}

\newcommand{\E}{\mathcal{E}}

\newcommand{\spt}{\mathrm{supp}}

\newcommand{\ud}{\mathrm{d}}

\newcommand{\Div}{\mathrm{Div}}

\newcommand{\Huno}{\mathcal H^1}

\newcommand{\Ccal}{{\mathcal{C}}}

\newcommand{\dist}{\mathrm{dist}}
\newcommand{\nsc}{\mathrm{nsc}}

\newcommand{\ii}{\mathrm{in}}
\newcommand{\cc}{\mathrm{c}}
\newcommand{\G}{\mathcal{G}}

\newcommand{\ce}{\ep}
\newcommand{\argmin}{\operatorname*{argmin}}

\def\red#1{\textcolor{red}{#1}}

\def\XXint#1#2#3{{\setbox0=\hbox{$#1{#2#3}{\int}$}
     \vcenter{\hbox{$#2#3$}}\kern-.5\wd0}}

\usepackage{latexsym}

\makeatletter
\def\@splitop#1#2\@nil{$\mathscr{#1}\!\!$\calligra#2\,\,}
\newcommand*\DeclareCursiveOperator[2]{%

 \newcommand#1{\mathop{\mbox{\@splitop#2\@nil}}\nolimits}}
\makeatother
\DeclareCursiveOperator{\Anew}{A}
\DeclareCursiveOperator{\Bnew}{B}
\DeclareCursiveOperator{\Cnew}{C}
\DeclareCursiveOperator{\Dnew}{D}
\DeclareCursiveOperator{\Enew}{E}
\DeclareCursiveOperator{\Qnew}{Q}
\DeclareCursiveOperator{\Tnew}{T}

\DeclareMathOperator{\curl}{curl}
\DeclareMathOperator{\Curl}{Curl}
\DeclareMathOperator{\CURL}{\textsc{curl}}
\DeclareMathOperator{\inc}{inc}
\DeclareMathOperator{\INC}{INC}
\DeclareMathOperator{\tr}{tr}

\allowdisplaybreaks

\title[Planar linearized elasticity with incompatible kinematics]{Variational formulation of planar linearized elasticity with incompatible kinematics
}
\author[Pierluigi Cesana]
{Pierluigi Cesana}
\address[Pierluigi Cesana]{Institute of Mathematics for Industry, Kyushu University, 744 Motooka, Fukuoka 819-0395, Japan}
\email[P. Cesana]{cesana@math.kyushu-u.ac.jp}
\author[Edoardo Fabbrini]
{Edoardo Fabbrini}
\address[Edoardo Fabbrini]{Graduate School of Mathematics, Kyushu University, 744 Motooka, Fukuoka 819-0395, Japan}
\email[E. Fabbrini]{fabbrini.edoardo.840@s.kyushu-u.ac.jp}
\author[Marco Morandotti]{Marco Morandotti}
\address[Marco Morandotti]{Dipartimento di Scienze Matematiche ``G.~L.~Lagrange'', Politecnico di Torino, Corso Duca degli Abruzzi, 24, 10129 Torino, Italy}
\email[M. Morandotti]{marco.morandotti@polito.it}

\date{\today}

\begin{document}

%%%% Keyword entries to be placed here %%%%

%%%% Abstract text to be placed here %%%%%%%%%%%%
\begin{abstract}
We present a variational characterization of mechanical equilibrium in the planar strain regime for systems with incompatible kinematics. For  non-simply connected domains, we show that the equilibrium problem for a non-liftable strain-stress pair can be reformulated as a well-posed minimization problem for the Airy potential of the system. 
We characterize kinematic incompatibilities on internal boundaries as rotational or translational mismatches, in agreement with Volterra's modeling of disclinations and dislocations. Finally, we establish that the minimization problem for the Airy potential can be reduced to a finite-dimensional optimization involving cell formulas.
\end{abstract}
%%%%%%%%%%%%%%%%%%%%%%%%%%%


%%%%%%%%%% Insert the texts which can accomdate on firstpage in the tag "fmtext" %%%%%

%\section{Insert A head here}
%%%%% Insert A head here
%
%This demo file is intended to serve as a ``starter file''
%for rsproca journal papers produced under \LaTeX\ using
%rsproca\_new.cls v1.0.
%
%\subsection{Insert B head here}
%%%%% Insert B head here
%Subsection text here.
%
%\subsubsection{Insert C head here}
%%%%% Insert C head here
%Subsubsection text here.
%
%\section{Equations}
%
%Sample equations.
%
%%%% Numbered equation
%\begin{align}\label{1.1}
%\begin{split}
%\frac{\partial u(t,x)}{\partial t} &= Au(t,x) \left(1-\frac{u(t,x)}{K}\right)-B\frac{u(t-\tau,x) w(t,x)}{1+Eu(t-\tau,x)},\\
%\frac{\partial w(t,x)}{\partial t} &=\delta \frac{\partial^2w(t,x)}{\partial x^2}-Cw(t,x)+D\frac{u(t-\tau,x)w(t,x)}{1+Eu(t-\tau,x)},
%\end{split}
%\end{align}
%
%\begin{align}\label{1.2}
%\begin{split}
%\frac{dU}{dt} &=\alpha U(t)(\gamma -U(t))-\frac{U(t-\tau)W(t)}{1+U(t-\tau)},\\
%\frac{dW}{dt} &=-W(t)+\beta\frac{U(t-\tau)W(t)}{1+U(t-\tau)}.
%\end{split}
%\end{align}
%
%%%%% Unnumbered equation
%\begin{eqnarray}
%\frac{\partial(F_1,F_2)}{\partial(c,\omega)}_{(c_0,\omega_0)} = \left|
%\begin{array}{ll}
%\frac{\partial F_1}{\partial c} &\frac{\partial F_1}{\partial \omega} \\\noalign{\vskip3pt}
%\frac{\partial F_2}{\partial c}&\frac{\partial F_2}{\partial \omega}
%\end{array}\right|_{(c_0,\omega_0)}\notag\\
%=-4c_0q\omega_0 -4c_0\omega_0p^2 =-4c_0\omega_0(q+p^2)>0.
%\end{eqnarray}

%%%%%%%%%%%%%%% End of first page %%%%%%%%%%%%%%%%%%%%%

\maketitle

\section{Introduction}
The study of compatibility conditions in elasticity  has played a fundamental role in the mathematical formulation of continuum mechanics, with roots tracing back to the mid-19th century through the works of Beltrami, Saint Venant, Michell, among others. These conditions guarantee that the mechanical strains within a material are consistent with the assumed displacement field, which is typically considered to be single-valued. 
Historically, these conditions were first examined in simply connected domains, such as
$\mathbb{R}^3$, where the elasticity theory was initially developed. 
In contrast, non-simply connected domains, which possess topological features like voids or holes, introduce significant complexity to the problem. In such domains,  the strain and displacement fields must be carefully examined to ensure that they are compatible with the topology of the material,  namely complementary loop integral conditions around each hole must be imposed in addition to the usual local compatibility condition.

The importance of understanding elasticity in non-simply connected settings has grown significantly with the emergence of modern applications, including the modeling of porous materials, fracture mechanics, additive manufacturing (3D printing), and materials with complex microstructures, such as metamaterials. These materials frequently exhibit internal loops, voids, and dislocation structures that demand compatibility conditions capable of accounting for nontrivial topology.

A systematic mathematical treatment of compatibility in non-simply connected domains has been developed in recent years, particularly through the work of Yavari and collaborators \cite{Angoshtari2016,Yavari2013,Yavari2020,Yavari12,yavari13}.
In the nonlinear setting, their analysis addresses compatibility conditions for both the deformation gradient $F$ and the right Cauchy-Green strain tensor  
$C = F^\top F$ \cite{Yavari2013} (see also \cite{Yavari2020}),
while related work by Acharya addressed compatibility conditions formulated in terms of the left Cauchy-Green tensor $B=FF^\top$  \cite{Acharya99a}. Further developments have extended these ideas to $L^2$ deformation fields in multiphase materials with internal voids, highlighting the interplay between compatibility conditions and material microstructure \cite{Angoshtari2016}. 

In this paper, we revisit the classical setting of linearized elasticity in two dimensions, focusing on non-simply connected planar domains.
In this context, the classical question of determining the necessary and sufficient conditions for the existence of a single-valued displacement field~$u$
such that the compatible strain tensor $\epsilon$ is given by $\nabla^{\sym}u=\frac12(\nabla u + (\nabla u)^\top)$  was first addressed by Michell in 1899 \cite{Michell}.
Our goal is to establish a rigorous variational formulation to characterize the mechanical equilibrium conditions in non-simply connected domains, by extending Michell's results to non-compatible strain fields.
Specifically, we identify the equations that characterize displacement fields that are not globally single-valued on the boundaries of the domain, when the system is in mechanical equilibrium.

We rigorously demonstrate that the violation of these compatibility conditions gives rise to precisely two types of kinematic incompatibility: translational and rotational. Importantly, we show that the translational incompatibility corresponds to the classical distortion induced by an edge dislocation, the canonical translational lattice defect introduced by Volterra in his seminal 1907 paper \cite{V07} (see also \cite{W68,RV92,Z97}; {see also \cite{Delphenich} for a recent translation of Volterra's paper into English}). The rotational incompatibility, on the other hand, corresponds to a wedge disclination, which reflects rotational irregularities caused by the failure of rotation closure around a loop in the crystal lattice in the undeformed configuration.

Let $\Omega\subset\R^2$ be a bounded, open, and simply connected domain, with $J$ dislocations and $K$ disclinations located at $x^1,\ldots,x^J\in\Omega$ and $y^1,\ldots,y^K\in\Omega$, respectively.
Let $\varepsilon>0$ be sufficiently small so that the disks $B_\varepsilon(x^j)$ and $B_\varepsilon(y^k)$ are mutually disjoint and all contained in~$\Omega$, let  $\sigma=\C\epsilon$ denote the  mechanical stress, and let $v$ denote the Airy stress function of the system, linked to $\sigma$ via the classical relations
$\sigma_{11}=v_{x_2x_2}$, $\sigma_{12}=\sigma_{21}= -v_{x_1x_2}$, $ \sigma_{22}=v_{x_1x_1}$\,.
Our main result, Theorem~\ref{2502282100}, establishes the following equivalence: the pair $\epsilon-\sigma$ solves
\begin{equation}\label{2503041822}
\begin{cases}
\Dive\sigma = 0 & \text{in $\Omega_{\varepsilon}$\,,}\\	
\sigma n = 0 & \text{on $\partial\Omega_{\varepsilon}$\,,} \\[2mm]
\curl\Curl\epsilon = 0 & \text{in $\Omega_{\varepsilon}$\,,} \\
\displaystyle \int_{\partial B_\varepsilon^k} (\epsilon_{rq,c}-\epsilon_{qc,r})\,\ud x_q = s^k  & \text{for $k = 1,\ldots, K$,} \\[3mm]
\displaystyle \int_{\partial B_\varepsilon^j} [\epsilon_{rc}-x_q(\epsilon_{rc,q}-\epsilon_{cq,r})]\,\ud x_c = b^j_r & \text{for $j = 1,\ldots, J$ and $r=1,2$}
\end{cases}
\end{equation}
where $s^k$ ($k=1,\ldots,K$) are the Frank angles of the disclinations and $b^j$ ($j=1,\ldots,J$) are the Burgers vectors of the dislocations, if and only if $v$ minimizes the functional $\I\colon H^2(\Omega_\varepsilon)\to\R$ defined by
\begin{equation*}\label{}
\begin{split}
\I(v)\coloneqq &\,\frac12\frac{1+\nu}{E} \int_{\Omega_{\varepsilon}} \big[|\nabla^2 v|^2-\nu(\Delta v)^2\big]\,\ud x+\sum_{j=1}^J %\frac{1}{2\pi\ce}
\ave_{\partial B_{\ce}^j} \langle\nabla v, \Pi(b^j)\rangle\,\ud\Huno
   +\sum_{k=1}^K %\frac{s^k}{2\pi\ce}
   \ave_{\partial B_\ce^k} s^k v\,\ud\Huno\,,
   \end{split}
\end{equation*}
where $\Omega_\varepsilon=\Omega\setminus\big((\cup_{j=1}^J \overline{B}_\varepsilon(x^j))\cup(\cup_{k=1}^K \overline{B}_\varepsilon(y^k))\big)$, in a suitably defined functions space that accounts for stress-free boundary conditions.
More specifically, the minimization of $\I$ is taken over a set of $H^2$-Sobolev functions, whose traces on the internal boundaries $\partial B_\varepsilon^j$ and $\partial B_\varepsilon^k$ are parametrized by coefficients of undetermined affine functions. These trace conditions ensure zero normal stress on these internal boundaries. This set was introduced in \cite{CDLM2024}, and these boundary conditions are fully consistent with Michell’s result \cite{Michell}.

This equivalence establishes a direct correspondence between the mechanical equilibrium conditions expressed in terms of stress and strain tensors and a variational minimization problem formulated in terms of the Airy stress function.

This functional $\mathcal{I}$ consists of two parts: a bulk term, representing the elastic energy stored in the system; and a finite sum of surface integrals, that quantifies the work required to create kinematic incompatibilities at the boundaries of the holes.
We emphasize that, in line with classical defect theory, Volterra disclinations and dislocations emerge from~$\mathcal{I}$ only in the limit as $\varepsilon \to 0$. 
In this asymptotic limit, the forcing term for disclinations converges to a Dirac delta supported at the center of the disclination core, modulated by the Frank angle~$s^k$.
Similarly, the forcing term for dislocations converges to a term proportional to partial derivatives of a Dirac delta, modulated by the Burgers vector~$b^j$.
Consequently, the model $\mathcal{I}$ studied in this paper for fixed $\varepsilon > 0$ should be regarded as a \textit{regularized} model, where the singularities are replaced by smooth boundary integrals, and the fully singular defect structure emerges only in the limit for $\varepsilon\to0$.

This regularization framework, know in the literature as the \emph{core radius approach}, is consistent with the earlier asymptotic analysis of Cermelli and Leoni for dislocations \cite{CermelliLeoni06}, later generalized by \cite{CDLM2024} for disclinations, and is directly inspired by the extensive literature on Ginzburg--Landau vortices, including the work of Bethuel, Brezis, and H\'elein \cite{BBH1994}.
We emphasize that the core radius approach is not merely a mathematical abstraction (see \cite{Peierls40}, see also the discussion in the introduction to \cite{ContiGarroniOrtiz15}). The core represents a region surrounding the dislocation line where the material undergoes significant distortion and where classical elasticity theory breaks down. The core can be experimentally observed and its size is typically estimated on the order of 1 to 10 Burgers vectors, which corresponds to a range of a few nanometers.

In the system \eqref{2503041822}, the first two equations describe the conservation of linear momentum in the absence of body forces, for a material with zero boundary stress. The final three conditions govern the kinematics, characterizing the displacement and strain incompatibilities. In particular, when 
$s^k=0$, $b^j=0$, these conditions reduce to the classical compatibility conditions derived by Yavari \cite{Yavari2013}, which are necessary and sufficient for the existence of a single-valued displacement field in non-simply connected domains $\Omega_\varepsilon$\,.
We further note that, in the case of simply connected domains, the final two line integrals disappear, and the system reduces to the classical Saint-Venant compatibility condition, $\curl\Curl\epsilon = 0$, that is necessary and sufficient to guarantee kinematic compatibility in the absence of defects.

We emphasize that, in this paper, we specifically focus on incompatibility arising from violations of single-valuedness of the displacement field along the boundaries of holes, 
$\partial B_{\ce}^j$ or $\partial B_{\ce}^k$.
As a result, the material occupying the open set 
$\Omega_{\varepsilon}$  is assumed to be free of bulk defects.
For a complementary analysis addressing incompatibility due to bulk disclinations and dislocations, we refer to \cite{CDLM2024}.

Finally, we highlight that, relying on our variational formulation and leveraging on the linearity of the problem and on some properties of the Monge-Amp\`{e}re operator, we characterize the minimum point of the functional $\I$ by explicit cell formulas, which result from an auxiliary finite-dimensional minimization problem. 
This has the two-fold advantage of complying with homogeneous Neumann boundary conditions in the internal boundaries and of reducing the computational complexity.

The paper is structured as follows. In Section \ref{sec_prel}, we provide background on kinematic compatibility in linearized elasticity and the modeling of disclinations and dislocations. Our main results are presented in Section \ref{202503121428}. We conclude the paper with an Appendix containing technical results and proofs.

\section{Preliminaries}\label{sec_prel}
We recall the standard framework of linearized elasticity, considering infinitesimal kinematics in the planar strain regime. Throughout this work, our reference domain is given~by
\begin{equation}\label{eq_Omega}
\text{$\Omega\subset\R^2$, a bounded, open, simply connected set.}
\end{equation}
Let $u\colon \Omega\to\mathbb{R}^2$ be the displacement field. 
Let $\epsilon\coloneqq\nabla^{\sym} u=\frac{1}{2}(\nabla u+(\nabla u)^\top)$ be the $2\times2$  symmetric strain tensor and let $\sigma$ be the $2\times2$ symmetric  stress tensor; we will consider a linear dependence of $\sigma$ on $\epsilon$, in such a way that the application $(\epsilon,\sigma)\mapsto \sigma:\epsilon$ be a bilinear, positive-definite, symmetric form.
One way to model this is to consider the operator 
$$\R^{2\times2}\ni m\mapsto\underline{m}\coloneqq
\begin{pmatrix}
m_{11} & m_{12} & 0\\
m_{21} & m_{22} & 0\\
0 & 0 & 0
\end{pmatrix}\in \R^{3\times3}$$
that completes a $2\times2$ matrix into a $3\times3$ matrix, and to consider the elasticity tensor $\mathbb{C} \colon \mathbb{R}^{3\times3}_{\sym} \to \mathbb{R}^{3\times3}_{\sym}$ of three-dimensional linearized elasticity with the request that it enjoys the major and minor symmetries for $\underline{\sigma}:\underline{\epsilon} \coloneqq \C\underline{\epsilon}:\underline{\epsilon}$ to be a bilinear, positive-definite, symmetric form.

Throughout the paper, we will make the assumption that whenever we write $\C m$ (with $m\in\R^{2\times2}$) we mean $\C\underline{m}$. In the context of plane strain elasticity, though, there is a well-known expression for $\sigma\in\R^{2\times2}$ as a function of $\epsilon\in\R^{2\times2}$, which is the following
\begin{subequations}\label{eq_epsilonsigma}
\begin{equation}\label{eq_epsilonsigma-1}
\sigma=\mathbb{C}\epsilon=\frac{E\nu}{(1-\nu)(1-2\nu)}\cof\epsilon+\frac{E}{1-2\nu}\epsilon
\end{equation}
together with its inverse
\begin{equation}\label{eq_epsilonsigma-2}
\epsilon=\mathbb{C}^{-1}\sigma=\frac{1-\nu^2}{E}\sigma-\frac{(1+\nu)\nu}{E}\cof\sigma\,,
\end{equation}
\end{subequations}
where $E>0$ is the Young's modulus and $\nu\in(-1,1/2)$ is the Poisson's ratio. 
In \eqref{eq_epsilonsigma}, the cofactor operator $\cof\colon\R^{2\times2}\to\R^{2\times2}$ acts on $2\times2$ matrices in the following way
$$\begin{pmatrix}
m_{11} & m_{12} \\
m_{21} & m_{22}
\end{pmatrix}=m\mapsto \cof(m)=
\begin{pmatrix}
m_{22} & -m_{21} \\
-m_{12} & m_{11}
\end{pmatrix},$$
and notice that the operations in \eqref{eq_epsilonsigma} preserve the symmetry of the $2\times2$ matrices involved.
Still in the context of plane strain elasticity, the Poisson effect might arise so that we will always consider
\begin{equation}\label{eq_sigmaPoisson}
\C\underline{\epsilon}\eqqcolon\underline{\sigma} = \begin{pmatrix}
\sigma_{11} & \sigma_{12} & 0 \\
\sigma_{12} & \sigma_{22} & 0 \\
0 & 0 & \sigma_{33}
\end{pmatrix},
\end{equation}
where $\sigma_{33}\neq 0$. 
To summarize, the matrices we will deal with (including their $3\times3$ extensions) are 
$$
\epsilon=\begin{pmatrix}
\epsilon_{11} & \epsilon_{12} \\
\epsilon_{12} & \epsilon_{22}
\end{pmatrix}
\,\,\text{and}\,\,
\sigma=\begin{pmatrix}
\sigma_{11} & \sigma_{12} \\
\sigma_{12} & \sigma_{22}
\end{pmatrix};
\underline{\epsilon}=\begin{pmatrix}
\epsilon_{11} & \epsilon_{12} & 0\\
\epsilon_{12} & \epsilon_{22} & 0\\
0 & 0 & 0
\end{pmatrix}
\,\,\text{and}\,\,
\underline{\sigma}=\begin{pmatrix}
\sigma_{11} & \sigma_{12} & 0\\
\sigma_{12} & \sigma_{22} & 0\\
0 & 0 & \sigma_{33}
\end{pmatrix}
$$
and we will write $\sigma=\C \epsilon$ and $\underline{\sigma}=\C \underline{\epsilon}$ indifferently (only in the latter case we will consider the presence of $\sigma_{33}$).
Notice that the equality 
$\sigma:\epsilon = \underline{\sigma}:\underline{\epsilon}$
holds true, as an immediate computation reveals.

The elastic energy for a planar body in planar strain regime reads
\begin{equation}\label{eq_energysigma}
\mathcal{F}(\sigma;\Omega)\coloneqq \frac{1}{2}\frac{1+\nu}{E} \int_\Omega \big[|\sigma|^2-\nu(\tr(\sigma))^2\big]\,\ud x\,.
\end{equation}
If the body is in mechanical equilibrium in absence of body forces, then the \emph{equilibrium equation}, i.e., the balance of linear momentum, $\Div\sigma=0$ in $\Omega$ must be satisfied.
The Airy stress function method assumes the existence of the \emph{Airy potential} or \emph{Airy stress function} $v\colon\Omega\to\mathbb{R}$ in terms of which the stress can be written as
\begin{subequations}\label{eq_Airyoperatorgen}
\begin{equation}\label{eq_Airyoperator1}
\sigma=\sigma[v]=\Airy(v),
\end{equation}
where $\Airy\colon\mathcal{C}^k(\Omega)\to\mathcal{C}^{k-2}(\Omega;\mathbb{R}^{2\times2}_{\sym})$ is defined by
\begin{equation}\label{eq_Airyoperator}
\Airy(v)\coloneqq\cof(\nabla^2 v)=\begin{pmatrix}
v_{x_2x_2} & -v_{x_1x_2} \\
-v_{x_2x_1} & v_{x_1x_1}
\end{pmatrix}.
\end{equation}
\end{subequations}
The main advantage of introducing such a function $v$ is that, by virtue of \eqref{eq_Airyoperator}, the equilibrium equation 
\begin{equation}\label{eq_Div_sigma}
\Div\sigma[v]=\Div(\cof(\nabla^2 v))=0
\end{equation}
is automatically satisfied as an identity at all points $x\in\Omega$.
Moreover, the elastic energy \eqref{eq_energysigma} is now expressed as 
\begin{equation}\label{eq_energyv}
\mathcal{G}(v;\Omega) \coloneqq \mathcal{F}(\sigma[v];\Omega)=\frac12\frac{1+\nu}{E} \int_\Omega \big[|\nabla^2 v|^2-\nu(\Delta v)^2\big]\,\ud x.
\end{equation}

\subsection{Compatibility in non-simply connected domains: stress-strain formulation}
For a displacement $u\colon\Omega\to\mathbb{R}^2$ of class $\mathcal{C}^3$, we have that 
\begin{equation}\label{eq_SVincstrong}
\inc\epsilon\coloneqq \curl\Curl\epsilon=0\,,
\end{equation}
as it can be easily verified by a simple computation.
Here recall for a vector field $V=(V^1,V^2)\colon\Omega\to\R^2$, we define $\curl V\coloneqq \partial_{x_1}V^2-\partial_{x_2} V^1$, and for a $2\times2$ matrix $M$, we define $\Curl M\coloneqq (\curl M_1,\curl M_2)$, $M_r$ being the $r$-th row of $M$.\footnote{
An alternative way to carry out the computations is to perform them on the trivial extension $\underline{u}$ of $u$ to three dimensions and use the $\CURL$ operator defined row-wise on $3\times3$ matrices. First, let $(x_1,x_2,x_3)\mapsto \underline{u}(x_1,x_2,x_3)\coloneqq (u^1(x_1,x_2),u^2(x_1,x_2),0)$, then construct the corresponding $\underline{\epsilon}=\nabla^{\sym}\underline{u}\in\mathbb{R}^{3\times3}$, and finally apply twice $\CURL$ in the sense of matrices. In formulas, we have, for a matrix $M\in\mathbb{R}^{3\times3}$, that $(\CURL M)_{rs} \coloneqq \varepsilon_{rpm}\partial_{x_p}M_{sm}$ ($\epsilon_{rpm}$ being the Levi-Civita alternating symbol), revealing that $(\CURL M)_{rs}=(\CURL M_s)^r$.
Therefore, in our case we have
$$\CURL \underline{\epsilon}=\begin{pmatrix}
0 & 0 & \curl \epsilon_1 \\
0 & 0 & \curl \epsilon_2 \\
0 & 0 & 0
\end{pmatrix}\qquad\text{and}\qquad
\INC\underline{\epsilon}\coloneqq\CURL\CURL\underline{\epsilon}=\begin{pmatrix}
0 & 0 & 0 \\
0 & 0 & 0 \\
0 & 0 & \curl\Curl\epsilon
\end{pmatrix},$$
showing that, in the two-dimensional case, there holds $\inc \epsilon=(\INC\underline{\epsilon})_{33}$.}
The operator $\inc$ is called the \emph{incompatibility operator} and returns the zero value when the displacement is \emph{compatible}, i.e., it is single-valued.
The celebrated \emph{Saint-Venant principle} \cite{SV1855} states the converse: if \eqref{eq_SVincstrong} holds, then there must exist a field $u\in\mathcal{C}^3(\Omega;\R^{2})$ such that $\epsilon=\nabla^{\sym}u$.
We refer the reader to \cite{ACGK,ciarlet05,Geymonat09} for extensions of this result to Sobolev spaces on simply connected domains~$\Omega$; we resume the result in the following proposition.
\begin{proposition}\label{prop_weakSV}
Let $\Omega\subset\R^2$ be as in \eqref{eq_Omega} and let $\epsilon\in L^2(\Omega;\R^{2\times 2}_{\sym})$. 
Then \eqref{eq_SVincstrong} holds in $H^{-2}(\Omega)$ if and only if there exists $u\in H^1(\Omega;\R^2)$ such that $\epsilon =\nabla^{\sym}u$. Moreover, $u$ is unique
up to rigid motions.
\end{proposition}

The Saint-Venant principle requires the simple connectedness of the domain $\Omega$ to work; at the very core of it, there is the application of the Poincaré Lemma, which cannot dispensate of this topological condition. 
To treat the case of non simply connected planar domains $\Omega\subset\R^2$, in which the $\inc=\curl\Curl$ operator cannot be inverted, additional compatibility conditions have been identified that are both necessary and sufficient to establish the Saint-Venant principle.
We are looking, here, at domains $\Omega^{\nsc}\subset\R^2{}$ of the kind
\begin{equation}\label{eq_Omega_nsc}
\Omega^{\nsc}=\Omega\setminus\bigcup_{i=1}^N \overline\Omega^i\,,
\end{equation}
where $\Omega\subset\R^2$ is as in \eqref{eq_Omega} 
%a simply connected domain 
and the open sets $\Omega^i\subset\Omega$ ($i=1,\ldots,N$) are the pairwise well disjoint (i.e., $\overline\Omega^{i_1}\cap\overline\Omega^{i_2}=\emptyset$ for $i_1\neq i_2$) ``holes'' that make the domain non simply connected. 

According to \cite[Proposition 2.8, Eqs. (2.44) and (2.45)]{Yavari2013}, the necessary and sufficient kinematic compatibility conditions that the strain $\epsilon$ must satisfy are:
\begin{enumerate}[(i)]
\item the Saint-Venant condition \eqref{eq_SVincstrong} in the bulk, namely
\begin{equation}\label{202502151537}
\inc \epsilon=0\qquad\text{in $\Omega^{\nsc}$;}
\end{equation}
\item the additional conditions on the internal boundaries
\begin{equation}\label{202308101409}
\begin{cases}
\displaystyle \int_{\gamma^i} \sum_{q=1}^2 (\epsilon_{rq,c}-\epsilon_{qc,r})\,\ud x_q=0 & \text{for $i=1,\ldots,N$, and $c,r=1,2%\beta_1(\Omega)
$.}\\[5mm]
\displaystyle \int_{\gamma^i} \sum_{c,q=1}^2 [\epsilon_{rc}-x_q(\epsilon_{rc,q}-\epsilon_{cq,r})]\,\ud x_c=0 & \text{for $i=1,\ldots,N$, and $r=1,2%\beta_1(\Omega)
$,}
\end{cases}
\end{equation}
\end{enumerate}
Regarding the first equation in \eqref{202308101409}, we observe that the integrand vanishes identically when $c=r=1$ and when $c=r=2$ (by the symmetry of $\epsilon$) and that the integrand for $c=1$ and $r=2$ is the opposite of that for $c=2$ and $r=1$; thus, out of the four equations contained in \eqref{202308101409}$_1$, only one is meaningful.
Notice that the second equation in \eqref{202308101409} actually contains two equations, corresponding to the two possible values of $r\in\{1,2\}$. 

To describe the mechanical equilibrium of the system, conditions 
\eqref{202502151537} and \eqref{202308101409} should be complemented by
\begin{enumerate}
\item[(iii)] $\Div\sigma=\rho$  in $\Omega^{\nsc}$,
\item[(iv)] $\sigma n=\sigma_o$ on $\partial\Omega^{\nsc},$
\end{enumerate}
where $n$ is the outer unit normal to $\partial\Omega^{\nsc}=\cup_{i=0}^N \partial\Omega^i\eqqcolon \cup_{i=0}^N \gamma^i$.
Here we have defined $\gamma^0=\partial\Omega^0\coloneqq\partial\Omega$; notice that $\partial\Omega^{\nsc}$ is the (disjoint) union of $N+1$ Jordan curves.

Conditions (i)--(iv) describe the equilibrium configurations of a body occupying a domain $\Omega^{\nsc}$    
subject to body forces $\rho\colon\Omega^{\nsc}\to\mathbb{R}^2$
and surface tension $\sigma_o\colon\partial\Omega^{\nsc}\to\mathbb{R}^2$.
In what follows, we will see that equilibrium configurations in our problems for incompatible kinematics are attained for $\rho=0$ and $\sigma_o=0$ and non-trivial solutions will emerge due to kinematic incompatibility violating~\eqref{202308101409}. 

We adopt the approach of intrinsic elasticity \cite{AMROUCHE2006116,CIARLET2004307,CIARLET2014_intrform}, formulating mechanical equilibrium conditions directly in terms of strains and stresses rather than in terms of the mechanical displacement.
This choice is natural in the present context for two main reasons. 
Firstly, the mechanical displacement is explicitly not required to be one-to-one due to the presence of kinematic incompatibility constraints; this makes it more convenient to formulate the problem directly in terms of strains and stresses, which are the physically observable quantities. 
Secondly, we are particularly interested in Neumann-type boundary conditions, where equilibrium is imposed directly in terms of stresses at the boundary, without requiring explicit access to the displacement field.
Nonetheless, as part of our main results in Section~\ref{202503121428}, we establish that once the inelastic strain component is isolated and removed, the remaining mechanical strain field can be lifted to a single-valued %\blu{elastic/compatible?}
 displacement field.

\subsection{Compatibility in non-simply connected domains: Airy potential formulation}
In his 1899 paper \cite{Michell}, Michell writes that for the displacement to be single-valued, i.e., kinematically compatible, the associated Airy function must satisfy the conditions (see \cite[formulas (11) and (12)]{Michell})
\begin{equation}\label{eq_Michell0}
\begin{cases}
\Delta^2 v=0 & \text{in $\Omega^{\nsc}$,}\\[2mm]
\displaystyle v(\lambda)=\int_0^\lambda (H_1^i\,\ud x_1+H_2^i\,\ud x_2)+a_1^i x_1+a_2^i x_2+a_0^i & \text{on $\gamma^i$, for $i=0,1,\ldots,N$,}\\[2mm]
\displaystyle \partial_n v=H_1^i\frac{\ud x_2}{\ud \lambda}-H_2^i\frac{\ud x_1}{\ud \lambda}+a_1^i\frac{\ud x_2}{\ud \lambda}-a_2^i\frac{\ud x_1}{\ud \lambda} & \text{on $\gamma^i$, for $i=0,1,\ldots,N$,}\\
\end{cases}
\end{equation}
where $\Delta^2\coloneqq\Delta\Delta$ is the bilaplacian operator, and the $3(N+1)$ compatibility conditions on the boundaries 
\begin{equation}\label{eq_Michell1}
\begin{cases}
\displaystyle \int_{\gamma^i} \partial_n(\Delta v)\,\ud \Huno=0 \\[3mm]
\displaystyle \int_{\gamma^i} \bigg(x_1\partial_t(\Delta v)-x_2\partial_n(\Delta v)+\frac{(\nabla^2 v\, t)_1}{1-\nu}\bigg)\,\ud \Huno=0 \\[3mm]
\displaystyle \int_{\gamma^i} \bigg(x_1\partial_n(\Delta v)+x_2\partial_t(\Delta v)+\frac{(\nabla^2 v\, t)_2}{1-\nu}\bigg)\,\ud \Huno=0 
\end{cases}
\quad\text{for $i=0,1,\ldots,N$.}
\end{equation}
In \eqref{eq_Michell0}, $\lambda\in[0,|\gamma^i|]$ for each of the $N+1$ equations in the second line,
where $H_1^i$ and $H_2^i$ are given functions on $\gamma^i$, and $a_0^i,a_1^i,a_2^i\in\R$ are undetermined for each $i=0,1,\ldots,N$. 
Notice that $(\ud x_2/\ud \lambda,-\ud x_1/\ud \lambda)=(n_1,n_2)$, so that the third equation in \eqref{eq_Michell0} can be written: $\partial_n v = (H_1^i+a_1^i) n_1 + (H_2^i+a_2^i) n_2$ on $\gamma^i$, for $i=0,1,\ldots,N$.

We set here the general notation for affine functions: the coefficients are denoted by $a_0^i,a_1^i,a_2^i\in\R$, so that 
$$a^i(x)=a^i(x_1,x_2)=a_0^i+\langle (a_1^i,a_2^i),x\rangle=a_0^i+a_1^ix_1+a_2^ix_2\,,
\qquad\text{for every $i=1,\ldots,N$.}$$

\begin{remark}
    We notice here the following facts:
    \begin{enumerate}
        \item There is the freedom to choose one triplet $a_0^{i^*}, a_1^{i^*}, a_2^{i^*}$ of the $3(N+1)$ constants $a_0^i, a_1^i, a_2^i$ in \eqref{eq_Michell0}, corresponding to a certain boundary $\gamma^{i^*}$, since the function $(x_1,x_2)\mapsto v(x_1,x_2)-a_1^{i^*}x_1-a_2^{i^*}x_2-a_0^{i^*}$ still solves \eqref{eq_Michell0}--\eqref{eq_Michell1}.
        \item The conditions in \eqref{eq_Michell1} are obtained from \cite[equations (6), (7), and (8)]{Michell} under the assumption that the body forces vanish. 
        The last summand in the second and third boundary integrals in \eqref{eq_Michell1} is obtained from \cite[equations (7) and (8)]{Michell} recalling that (see \eqref{eq_Airyoperatorgen} and also the second equation in \eqref{eq_formaltrans} below)
        \begin{equation}\label{eq_202409051159}
        (\sigma\, n)_1= (\nabla^2 v\, t)_2 
        \qquad\text{and}\qquad
        (\sigma\, n)_2= -(\nabla^2 v\, t)_1\,.
        \end{equation}
        This set of compatibility conditions is also given by Selvadurai in \cite[Theorem 8.8]{Selvadurai}.
        \item If $\Omega^{\nsc}=\Omega$ (it is simply connected), so that $\partial\Omega=\gamma^0$ ($N=0$), the corresponding set of conditions in \eqref{eq_Michell1} is trivially satisfied.
    \end{enumerate}
\end{remark}

In the next proposition, we prove that the boundary integrals in \eqref{202308101409} are the same as the boundary integrals in \eqref{eq_Michell1}, up to the constant $-E/(1-\nu^2)$.
\begin{proposition}\label{prop_202409051553}
Let $\Omega$ be as in \eqref{eq_Omega}, let $\Omega^1,\ldots, \Omega^N\subset\Omega$ be $N$ pairwise disjoint simply connected subsets with boundaries $\gamma^i\coloneqq \partial\Omega^i$.
Then, if $\epsilon=\C^{-1}\Airy(v)$, for $i=1,\ldots,N$,
\begin{subequations}\label{eq_YS}
\begin{eqnarray*}
    \displaystyle \frac{-E}{1-\nu^2} \int_{\gamma^i} \sum_{q=1}^2 (\epsilon_{1q,2}-\epsilon_{q2,1})\,\ud x_q &\!\!\!\! = &\!\!\!\! \int_{\gamma^i} \partial_n(\Delta v)\,\ud \Huno, \label{eq_YS_3}\\%[3mm]
    \displaystyle \frac{-E}{1-\nu^2} \int_{\gamma^i} \sum_{c,q=1}^2 [\epsilon_{2c}-x_q(\epsilon_{2c,q}-\epsilon_{cq,2})]\,\ud x_c &\!\!\!\! = &\!\!\!\! \int_{\gamma^i} \bigg(x_1\partial_n(\Delta v)+x_2\partial_t(\Delta v)+\frac{(\nabla^2 v\, t)_2}{1-\nu}\bigg)\,\ud \Huno, \label{eq_YS_2}\\%[3mm]
    \displaystyle \frac{-E}{1-\nu^2} \int_{\gamma^i} \sum_{c,q=1}^2 [\epsilon_{1c}-x_q(\epsilon_{1c,q}-\epsilon_{cq,1})]\,\ud x_c &\!\!\!\! = &\!\!\!\! \displaystyle \int_{\gamma^i} \bigg(x_1\partial_t(\Delta v)-x_2\partial_n(\Delta v)+\frac{(\nabla^2 v\, t)_1}{1-\nu}\bigg)\,\ud \Huno. \label{eq_YS_1}
\end{eqnarray*}
\end{subequations}
\end{proposition}
\begin{proof}
The thesis follows from a direct computation, recalling the relationship between~$\epsilon$ and~$\sigma$ (see \eqref{eq_epsilonsigma}) and that between $\sigma$ and $v$ (see \eqref{eq_Airyoperatorgen}).
The computations below hold for ever $\gamma^i$, $i=1,\ldots,N$, so we will denote by $\gamma$ the generic curve; we illustrate the computation for the second equality: letting $r=1$, we have 
\begin{align*}
&\,\frac{-E}{1-\nu^2}\int_{\gamma}\big\{[\epsilon_{11}-x_2(\epsilon_{11,2}-\epsilon_{12,1})] \,\ud x_1+ [\epsilon_{12}-y(\epsilon_{12,2}-\epsilon_{22,1})] \,\ud x_2\big\}\\
=&\, \frac{-1}{1-\nu}\int_{\gamma} \big\{ [(1-\nu) v_{x_2x_2}-\nu v_{x_1 x_1}-x_2((1-\nu) v_{x_2x_2x_2}-\nu v_{x_1 x_1 x_2}+v_{x_1 x_2 x_1})] \,\ud x_1  \\
& \qquad\qquad+[-v_{x_1 x_2}-x_2(-v_{x_1 x_2 x_2}-(1-\nu) v_{x_1 x_1 x_1}+\nu v_{x_2 x_2 x_1})] \,\ud x_2 \big\} \\
=&\, \frac{-1}{1-\nu}\int_{\gamma} \{-v_{x_1x_1}\,\ud x_1-v_{x_1x_2}\,\ud x_2\} %(1-\nu)
-\int_{\gamma} \big\{[\Delta v-x_2\partial_{x_2}(\Delta v)]\,\ud x_1+x_2\partial_{x_1}(\Delta v)\,\ud x_2 \big\}\\
=&\, \int_\gamma \frac{(\nabla^2 v\,t)_1}{1-\nu}\,\ud\Huno - %(1-\nu)
\int_\gamma x_2\partial_n(\Delta v)\,\ud\Huno +%(1-\nu)
\int_\gamma x_1\partial_t(\Delta v)\,\ud\Huno,
\end{align*}
where the last equality follows from using \eqref{202308152252} and \eqref{202308142315}. 
We have obtained the first relationship in the thesis.
Letting $r=2$ in the first equation in \eqref{202308101409}, and performing analogous computations we have the last equality.
Finally, letting $r=1$ and $c=2$ in \eqref{202308101409}, and performing once again the same computations, we obtain the first equality.
\end{proof}

\subsection{Modeling disclinations and dislocations via kinematic incompatibility}
The study of dislocations and disclinations has attracted  significant attention from both modelers \cite{Acharya19, Acharya15, CPL14, KTT2024, ZA2018, ZHANG18} and mathematicians \cite{AlicandroDeLucaGarroniPonsiglione14, BlassFonsecaLeoniMorandotti15,BlassMorandotti17,ContiGarroniMueller11,ContiGarroniMueller22, HudsonMorandotti2017,Mora2017}.
For a more comprehensive overview, including references on statistical theories of disclinations and dislocations, we refer readers to \cite{CDLM2024} and \cite{ZA2018}. Additional insights into their dynamics can be found in \cite{BonaschiVanMeursMorandotti17,CGMP2025,vMP2024}.
For more work in modeling topological defects in linearized elasticity, see also
\cite{AmstutzVanGoethem2016,vG2017,vaGoethemDupret2012}.

In this work, we adopt the couple stress theory framework \cite{Toupin1962,Toupin1964} to model disclinations and dislocations. This approach allows us to account for rotational deformations and moment stresses at small scales, effects that are not captured by classical elasticity 
\cite{FRESSENGEAS20113499,TAUPIN2015277}.

Equation \eqref{eq_SVincstrong} emerges then as measures of kinematical incompatibility: 
in a simply connected domain~$\Omega$,
the displacement $u$ is incompatible if and only if the right-hand side of \eqref{eq_SVincstrong} does \emph{not} vanish. 
In the presence of isolated dislocations and disclinations in the material, we have (invoking the Closed Graph Theorem, see the discussion in \cite[Section 1.2]{CDLM2024})
\begin{equation}\label{eq_inc}
\inc\epsilon=\curl\Curl\epsilon=\curl\alpha-\theta \qquad\text{in $\mathcal{D}'(\Omega)$},
\end{equation}
where $\alpha\in\ED(\Omega)$ and $\theta\in\WD(\Omega)$ are the measures concentrated on edge dislocations and wedge disclinations, respectively, and 
\[\begin{split}
\ED(\Omega)\coloneqq & \bigg\{\alpha=\sum_{j=1}^J b^j\de_{x^j}: J\in\N, b^j\in\R^2\setminus\{0\}, x^j\in\Omega, x^{j_1}\neq x^{j_2}\text{ for }j_1\neq j_2\bigg\},\\
\WD(\Omega)\coloneqq & \bigg\{\theta=\sum_{k=1}^K s^k\de_{y^k}: K\in\N, s^k\in\R\setminus\{0\}, y^k\in\Omega, y^{k_1}\neq y^{k_2}\text{ for }k_1\neq k_2\bigg\}.
\end{split}\]
We will always suppose that the supports of disclinations and dislocations are disjoint, $\spt(\alpha)\cap\spt(\theta)=\emptyset$, namely
\begin{equation}\label{eq_disj_supp}
\{x^j\}_{j=1}^J\cap\{y^k\}_{k=1}^K=\emptyset.
\end{equation}
For such $\alpha$ and $\theta$, \eqref{eq_inc} reads
\begin{subequations}\label{eq_incexplicit}
\begin{equation}\label{eq_incexplicit_1}
\inc\epsilon=\curl\Curl\epsilon= \curl\alpha
-\sum_{k=1}^K s^k\de_{y^k} \qquad\text{in $\mathcal{D}'(\Omega)$},
\end{equation}
where $\curl\alpha$ can be expressed in the following three equivalent ways
\begin{equation}\label{eq_incexplicit_2}
\curl\alpha=-\!\sum_{j=1}^J |b^j|\partial_{\frac{(b^j)^\perp}{|b^j|}} \de_{x^j}\equiv - \!\sum_{j=1}^J \big((b^j)^1\partial_{x_2}\de_{x^j}-(b^j)^2\partial_{x_1}\de_{x^j}\big)\equiv
-\! \sum_{j=1}^J (b^j\times\nabla)\de_{x^j}\,.
\end{equation}
\end{subequations}
Considering $v\colon\Omega\to\R$, defining $\sigma[v]=\Airy(v)$ by \eqref{eq_Airyoperatorgen}, and letting $\epsilon[v]\coloneqq\mathbb{C}^{-1}\sigma[v]$ (see \eqref{eq_epsilonsigma}), we have the formal transformations
\begin{equation}\label{eq_formaltrans}
\inc\epsilon[v]=\curl\Curl\epsilon[v]=\frac{1-\nu^2}{E}\Delta^2 v\qquad\text{and}\qquad \sigma[v]n=\Pi(\nabla^2v\, t),
\end{equation}
where $n$ and $t$ are the outer unit normal and unit tangent vector, respectively, to $\partial\Omega$, and %$\Delta^2\coloneqq\Delta\Delta$ is the bilaplacian operator \footnote{\textcolor{red}{Nota: bilaplaciano definito qui ma utilizzato anche in precedenza in \ref{eq_Michell0} }}.
%The vector 
$\sigma n$ evaluated on $\partial\Omega$ is the \emph{traction at the boundary}; here, we use $\Pi$ to denote the rotation by $-\pi/2$, that is $\Pi(w_1,w_2)=(w_2,-w_1)$, for every $w=(w_1;w_2)\in\R^2$.
In view of \eqref{eq_formaltrans}, we write the two equivalent formulations for the traction-free equilibrium problem in the presence of defects
\begin{equation}\label{eq_equilibriumproblems}
\begin{cases}
    \inc\epsilon=\curl\alpha-\theta & \text{in $\Omega$,}\\
    \Div \sigma=0 & \text{in $\Omega$,}\\
    \sigma n=0 & \text{on $\partial\Omega$,}
\end{cases}\qquad\text{and}\qquad
\begin{cases}
    \displaystyle \frac{1-\nu^2}{E}\Delta^2 v=\curl\alpha-\theta &\text{in $\Omega$,}\\[2mm]
    \nabla^2 v\, t=0 & \text{on $\partial\Omega$.}
\end{cases}
\end{equation}
It has been proved in \cite[Corollaries 1.5 and 1.9]{CDLM2024} that the two formulations
for the traction-free equilibrium problem in the presence of disclinations are equivalent.
 
From the linearity of the bilaplacian operator $\Delta^2$, we can additively decompose the Airy function $v$ into 
%\emph{elastic} and \emph{plastic} 
{\emph{compatible} and \emph{incompatible}} contributions,
$$v=v^\cc+v^\ii,$$ 
where the %plastic
{incompatible} contribution $v^\ii$ is given by the superposition of the Green's functions corresponding to a defect located at a point $x^j$ (in case of dislocations) or at a point $y^k$ (in case of a disclination).
Therefore, we write
\begin{equation}\label{eq_vp}
v^\ii(x)=\sum_{j=1}^J (-b^j \times v_D(x-x^j)) +\sum_{k=1}^K (-s^k v_d(x-y^k)) \eqqcolon \sum_{j=1}^J v_D^j(x)+\sum_{k=1}^K v_d^k(x),
\end{equation}
where $v_d\colon\R^2\to\R$ and $v_D\colon\R^2\to\R^2$ are the functions defined on $\R^2$ by
\begin{subequations}\label{eq_v_dD}
\begin{eqnarray}
v_d(x) &\!\!\!\! \coloneqq &\!\!\!\! \frac{E}{1-\nu^2}\frac{|x|^2}{16\pi}\log|x|^2\quad (x\neq0),\qquad v_d(0)=0,\label{eq_vd}\\
v_D(x) &\!\!\!\! \coloneqq &\!\!\!\! \frac{E}{1-\nu^2}\frac{x}{8\pi}\big(\log|x|^2+1\big)\quad (x\neq0)\qquad v_D(0)=0,\label{eq_vD}
\end{eqnarray}
\end{subequations}
see \cite[equations (1.21) and (3.32)]{CDLM2024}.
From \eqref{eq_v_dD}, it is immediate to verify that
$$\frac{1-\nu^2}{E}\Delta^2 v_d=\de_0
\qquad\text{and}\qquad 
\frac{1-\nu^2}{E}\Delta^2 (e\times v_D)=(e\times\nabla)\de_0\,,$$
for every unit vector $e\in\mathbb{S}^1$,
so that the functions $v_D^j$ and $v_d^k$ defined in \eqref{eq_vp} for all $j=1,\ldots,J$ and $k=1,\ldots,K$ satisfy
\begin{equation}\label{eq_202409041452}
\frac{1-\nu^2}{E}\Delta^2 v_D^j=-(b^j\times\nabla)\de_{x^j}
\qquad\text{and}\qquad
\frac{1-\nu^2}{E}\Delta^2 v_d^k=-s^k\de_{y^k}\,.
\end{equation}
We can now define, via \eqref{eq_Airyoperatorgen} and \eqref{eq_epsilonsigma}, the %plastic
{incompatible} stress and strain
\begin{equation}\label{eq_sigmap_from_vp}
\sigma^\ii=\sigma^\ii[v^\ii]\coloneqq\cof(\nabla^2 v^\ii),\qquad
\epsilon^\ii=\epsilon^\ii[v^\ii]\coloneqq \mathbb{C}^{-1}\sigma^\ii[v^\ii]=\mathbb{C}^{-1}\cof(\nabla^2 v^\ii),
\end{equation}
and notice that, by \eqref{eq_202409041452}, the %plastic
{incompatible} strain $\epsilon^\ii$ solves \eqref{eq_incexplicit}.

\subsection{Core radius approach}
The core radius approach to study the equilibrium configuration for a defected planar body with dislocations and disclinations is set up by cutting away small disks (the cores) centered at the defects. 
We choose the radius $\varepsilon>0$ of the cores $B_\varepsilon^j=B_\varepsilon(x^j)$ ($j=1,\ldots,J$) and $B_\varepsilon^k=B_\varepsilon(y^k)$ ($k=1,\ldots,K$) centered at the defects with the only constraint that neither any two cores overlap nor a core crosses the boundary $\partial\Omega$ of the domain. 
Collectively, we let $N\coloneqq J+K$ and we let
\begin{equation}\label{eq_defects}
\defects=\{\xi^i\}_{i=1}^N\coloneqq \{x^j\}_{j=1}^J\cup\{y^k\}_{k=1}^K
\end{equation} 
denote the ensemble of all defects in the body.
The corresponding cores will be denoted by $B_\varepsilon^i=B_\varepsilon(\xi^i)$, for $\xi^i\in\defects$. 
Moreover, for $i=1,\ldots,N$, we set, for simplicity of notation,
\begin{equation}\label{eq_sb_general}
\begin{cases}
b^i\coloneqq 0 & \text{if $\xi^i=y^k$ for some $k=1,\ldots,K$,}\\
s^i\coloneqq 0 & \text{if $\xi^i=x^j$ for some $j=1,\ldots,J$,}
\end{cases}
\end{equation}
and we define the extended measures of defects $\tilde\alpha\in\ED(\Omega)$ and $\tilde\theta\in\WD(\Omega)$ as
\begin{equation}\label{eq_extended_measures}
\tilde\alpha\coloneqq\sum_{i=1}^N b^i\de_{\xi^i}\,,\qquad
\tilde\theta\coloneqq\sum_{i=1}^N s^i\de_{\xi^i}\,.
\end{equation}
Notice that, since $b^i=0$ if $\xi^i\notin\{x^j\}_{j=1}^J$\,, namely if $\xi^i$ is not a dislocation, then $\tilde\alpha(\Omega)=\alpha(\Omega)$; likewise, since $s^i=0$ if $\xi^i\notin\{y^k\}_{k=1}^K$\,, namely if $\xi^i$ is not a disclination, then $\tilde\theta(\Omega)=\theta(\Omega)$.
Finally, notice that, owing to \eqref{eq_disj_supp} and to \eqref{eq_sb_general}, we have $\spt(\tilde\alpha)\cap\spt(\tilde\theta)=\emptyset$.
With the positions \eqref{eq_defects}, \eqref{eq_sb_general}, and \eqref{eq_extended_measures}, we fix $\varepsilon\in(0,\varepsilon_0)$, where
\begin{equation}\label{eq_varepsilon0}
\varepsilon_0\coloneqq \min\bigg\{%&
\frac12\min\Big\{|\xi^{i_1}-\xi^{i_2}|: \xi^{i_1},\xi^{i_2}\in\defects, i_1\neq i_2\Big\}, \min \big\{\dist(\xi,\partial\Omega):\xi\in\defects\big\}\bigg\}
\end{equation}
and we consider the perforated domain 
\begin{equation}\label{eq_Omega_eps}
\Omega_\varepsilon=\Omega_\varepsilon(\defects)\coloneqq\Omega\setminus\bigg(\bigcup_{i=1}^N\overline{B}_\varepsilon^i\bigg),\quad\text{with}\quad \partial \Omega_\varepsilon = \partial \Omega\cup\bigg(\bigcup_{j=1}^J \partial B_\varepsilon^j\bigg)\cup \bigg(\bigcup_{k=1}^K \partial B_\varepsilon^k\bigg).
\end{equation}
Notice that $\Omega_\epsilon$ is a bounded, connected but not simply connected, open subset of $\R^2$, whose boundary $\partial \Omega_\varepsilon$ is made of $N+1=J+K+1$ connected $1$-dimensional components. 

We now study the properties of $v^\ii$ and highlight the effects of the defects on the boundaries of the cores $B_\varepsilon^i$ ($i=1,\ldots,N$).
\begin{proposition}[properties of $v^\ii$]\label{prop_boundary_cond_incomp}
Let $\Omega\subset\R^2$ be as in \eqref{eq_Omega}, and let $\alpha\in\ED(\Omega)$ and $\theta\in\WD(\Omega)$ be such that \eqref{eq_disj_supp} holds true.
Let $\defects$ be as in \eqref{eq_defects} and let $\Omega_\varepsilon(\defects)$ be defined as in \eqref{eq_Omega_eps}; let $\gamma^i\coloneqq\partial B_\varepsilon^i$.
The function $v^\ii$ defined in~\eqref{eq_vp} enjoys the following properties: 
\begin{enumerate}
\item[(i)] it solves, in $\mathcal{D}'(\Omega)$, the equation
\begin{equation}\label{eq_bilaplacianappendix}
\frac{1-\nu^2}{E}\Delta^2 v= \curl\tilde\alpha-\tilde\theta=-\sum_{i=1}^N (b^i\times\nabla)\de_{\xi^i}-\sum_{i=1}^N s^i\de_{\xi^i}\,,
\end{equation}
where the measures $\tilde\alpha\in\ED(\Omega)$ and $\tilde\theta\in\WD(\Omega)$ are defined in \eqref{eq_extended_measures}; 
\item[(ii)] it verifies
\begin{subequations}\label{eq_necessary_incompatible}
\begin{equation}\label{eq_necessary_incompatible_bulk}
\frac{1-\nu^2}{E}\Delta^2 v^\ii=0 \qquad\text{in $\Omega_\varepsilon(\defects)$;}
\end{equation}
\item[(iii)] for every $i=1,\ldots,N$, it satisfies 
\begin{equation}\label{eq_necessary_incompatible_boundary}
\begin{cases}
\displaystyle \frac{1-\nu^2}{E} \int_{\gamma^i} \partial_n(\Delta v^\ii)\,\ud\Huno=\tilde\theta(\Omega^i)\,,\\[3mm]
\displaystyle \frac{1-\nu^2}{E} \int_{\gamma^i} \bigg(x_1\partial_t(\Delta v^\ii)-x_2\partial_n(\Delta v^\ii)+\frac{(\nabla^2 v^\ii\,t)_1}{1-\nu}\bigg)\,\ud\Huno=(\tilde\alpha(\Omega^i))_1\,,\\[3mm]
\displaystyle \frac{1-\nu^2}{E} \int_{\gamma^i} \bigg(x_1\partial_n(\Delta v^\ii)+x_2\partial_t(\Delta v^\ii)+\frac{(\nabla^2 v^\ii\,t)_2}{1-\nu}\bigg)\,\ud\Huno=(\tilde\alpha(\Omega^i))_2\,,
\end{cases}
\end{equation}
where $n$ is the outer unit normal to $\partial\Omega_\varepsilon(\defects)$ and $t$ points clockwise.
\end{subequations}
\end{enumerate}
\end{proposition}
The proof follows from a direct computation and is presented in Appendix~\ref{app_B}.

The following proposition translates the properties of $v^\ii$ listed in Proposition~\ref{prop_boundary_cond_incomp} to the associated strain $\epsilon^\ii$.
\begin{proposition}[properties of $\epsilon^\ii$]\label{prop_properties_of_eps_p}
Let the hypotheses of Proposition~\ref{prop_boundary_cond_incomp} hold. 
Then the strain $\epsilon^\ii$ defined in \eqref{eq_sigmap_from_vp} enjoys the following properties:
\begin{enumerate}
\item[(i)] it solves, in $\mathcal{D}'(\Omega)$, the equation
\begin{equation}\label{eq_inc-epsilon_full}
\inc\epsilon=\curl\Curl \epsilon= \curl\tilde\alpha-\tilde\theta=-\sum_{i=1}^N (b^i\times\nabla)\de_{\xi^i}-\sum_{i=1}^N s^i\de_{\xi^i}\,, 
\end{equation}
where the measures $\tilde\alpha\in\ED(\Omega)$ and $\tilde\theta\in\WD(\Omega)$ are defined in \eqref{eq_extended_measures}; 
\item[(ii)] it verifies
\begin{subequations}\label{eq_necessary_incompatible_epsilon}
\begin{equation}\label{eq_necessary_incompatible_bulk_epsilon}
\inc\epsilon^\ii=\curl\Curl\epsilon^\ii=0 \qquad\text{in $\Omega_\varepsilon(\defects)$;}
\end{equation}
\item[(iii)] for every $i=1,\ldots,N$, it satisfies 
\begin{equation}\label{eq_necessary_incompatible_boundary_epsilon}
\begin{cases}
\displaystyle \int_{\gamma^i} \sum_{q=1}^2 \big[(\epsilon^\ii)_{1q,2}-(\epsilon^\ii)_{q2,1}\big]\,\ud x_q=\tilde\theta(\Omega^i)\,,\\[4mm]
\displaystyle \int_{\gamma^i} \sum_{c,q=1}^2 \big[(\epsilon^\ii)_{rc}-x_q\big((\epsilon^\ii)_{rc,q}-(\epsilon^\ii)_{cq,r}\big)\big]\,\ud x_c= (\tilde\alpha(\Omega^i))_r\,,\quad\text{for $r=1,2$,}
\end{cases}
\end{equation}
\end{subequations}
where the boundaries $\gamma^i$ are oriented counter-clockwise.
\end{enumerate}
\end{proposition}
\begin{proof}
The proof of \eqref{eq_inc-epsilon_full} and \eqref{eq_necessary_incompatible_bulk_epsilon} is an immediate consequence of~\eqref{eq_formaltrans}, and \eqref{eq_bilaplacianappendix} and  \eqref{eq_necessary_incompatible_bulk} from Proposition~\ref{prop_boundary_cond_incomp}; the proof of \eqref{eq_necessary_incompatible_boundary_epsilon} follows from \eqref{eq_epsilonsigma}, \eqref{eq_necessary_incompatible_boundary} from Proposition~\ref{prop_boundary_cond_incomp}, and  Proposition~\ref{prop_202409051553}. 
\end{proof}

\section{Main results}\label{202503121428}
In this section we present our results establishing the existence of solutions for the mechanical equilibrium problem in an elastic material containing prescribed rotational and translational incompatibilities, %along its internal rims, 
subject to homogeneous Neumann boundary conditions (see the system in \eqref{2503041822}). 
Specifically, Theorem~\ref{2502282100} below demonstrates the equivalence between this equilibrium problem (referred to as the stress-strain mechanical equilibrium formulation) 
and the minimization problem for a functional defined in terms of the Airy potential (see \eqref{eq_energyperforated} below).
Furthermore, we clarify that these solutions, whether expressed through the stress-strain pair or the corresponding Airy potential, can be lifted in a specific sense, that is, after removing an inelastic term that solely accounts for the incompatibilities, they yield the existence of a single-valued displacement field. %(see Theorem~\ref{thm_liftabilityE} below).

Recalling \eqref{eq_energyv}, we define the energy functional $\I(\cdot;\Omega_\varepsilon)\colon H^2(\Omega)\to\R$ as
\begin{equation}\label{eq_energyperforated}
\begin{split}
\I(v;\Omega_{\varepsilon})\coloneqq &\,
\G(v;\Omega_{\varepsilon})    +\sum_{j=1}^J\frac{1}{2\pi\ce}\int_{\partial B_{\ce}^j} \langle\nabla v, \Pi(b^j)\rangle\,\ud\Huno
   +\sum_{k=1}^K \frac{s^k}{2\pi\ce}\int_{\partial B_\ce^k} v\,\ud\Huno\,,
   \end{split}
\end{equation}
which we seek to minimize in the class
\begin{equation}\label{eq_competitorsperforatedfloat}
\begin{split}
\Cnew(\Omega_\varepsilon(\defects))\coloneqq 
\big\{v\in H^2_0(\Omega): \text{$v=a^i$ in $B_\ce^i$\,, for some affine functions $a^i$, $i=1,\dots,N$}\}\,.
\end{split}
\end{equation}
This class of functions provides the appropriate framework for identifying mechanical equilibrium solutions. It was first introduced in \cite{CDLM2024} and encodes unknown affine boundary conditions on the internal boundaries.  
It was proved in \cite[Proposition~A.2]{CDLM2024} that these boundary conditions characterize the fact that \begin{equation}\label{eq_hesstang0}
\nabla^2 v\,t=0\qquad \text{on $\partial\Omega_\varepsilon$,} 
\end{equation}
provided that $\partial\Omega\in\mathcal{C}^4$.
Owing to \eqref{eq_202409051159}, this condition ensures zero normal stress on $\partial\Omega_\varepsilon(\defects)$, and was first recognized by Michell \cite{Michell}, see \eqref{eq_Michell0}.
We notice that affine functions $v$ have zero Hessian, so that they induce zero stress $\sigma$ and hence correspond to rigid body motions.

The following is the main theorem of this work.
\begin{theorem}\label{2502282100}
Let $\Omega$ be as in \eqref{eq_Omega} and let us assume that $\partial \Omega$ is smooth. 
Let $\alpha\in\ED(\Omega)$ and $\theta\in\WD(\Omega)$ satisfy \eqref{eq_disj_supp}, and let $\defects$ be as in \eqref{eq_defects}.
Let $\I$ and $\Cnew(\Omega_\varepsilon(\defects))$ be defined as in \eqref{eq_energyperforated} and \eqref{eq_competitorsperforatedfloat}, respectively.
Then
\begin{enumerate}
\item[(I)]
the strain $\hat\epsilon\in \mathcal{C}^\infty(\Omega_\varepsilon(\defects);\R^{2\times2}_{\sym})$  is the unique solution to 
\begin{equation}\label{2502131927}
\begin{cases}
\Dive\C\epsilon = 0 & \text{in $\Omega_{\varepsilon}$\,,}\\			
\C\epsilon n = 0 & \text{on $\partial\Omega_{\varepsilon}$\,,} \\%[2mm]
\curl\Curl\epsilon = 0 & \text{in $\Omega_{\varepsilon}$\,,} \\
\displaystyle \int_{\partial B_\varepsilon^i} (\epsilon_{rq,c}-\epsilon_{qc,r})\,\ud x_q = \tilde\theta(B_\varepsilon^i) & \text{for $i = 1,\ldots, N$,}\\[3mm]
\displaystyle \int_{\partial B_\varepsilon^i} [\epsilon_{rc}-x_q(\epsilon_{rc,q}-\epsilon_{cq,r})]\,\ud x_c = (\tilde\alpha(B_\varepsilon^i))_r & \text{for $i = 1,\ldots, N$ and $r=1,2$,}
\end{cases}
\end{equation}
where the boundaries $\partial B_\varepsilon^i$ are oriented counter-clockwise,
if and only if $\hat{v}\in \Cnew(\Omega_\varepsilon(\defects))\cap \mathcal{C}^\infty(\overline{\Omega}_\varepsilon(\defects))$ is the unique solution to 
\begin{equation}\label{2502131930}
\min\big\{\I(v;\Omega_{\varepsilon}(\defects)):v\in \Cnew(\Omega_\varepsilon(\defects))\big\},
\end{equation}
where $\hat\epsilon$ and $\hat{v}$ are related by
\begin{equation}\label{eq_cofattore}
\C\hat\epsilon=\cof(\nabla^2 \hat{v}).
\end{equation}
In other words, $[\hat\epsilon$ solves \eqref{2502131927} $\Leftrightarrow$ $\hat{v}$ solves \eqref{2502131930}$]$ $\Leftrightarrow$ \eqref{eq_cofattore} holds.
\item[(II)]
The strain and the Airy potential found in (I) are liftable in the following sense. There exists a unique (up to rigid motions) displacement $\hat{u}^\cc\in \mathcal{C}^\infty(\overline{\Omega}_\varepsilon(\defects);\R^{2})$ such that $\hat{\epsilon}^\cc=\nabla^\sym \hat{u}^\cc$, where $\hat\epsilon^\cc$ is defined by $\hat{\epsilon}^\cc\coloneqq \C^{-1}[\Airy(\hat{v})-\sigma^\ii]=\C^{-1}[\cof(\nabla^2 \hat{v})-\sigma^\ii] $, for $\sigma^\ii=\Airy(v^\ii)=\cof(\nabla^2 v^\ii)$, as in \eqref{eq_Airyoperatorgen}.
\end{enumerate}
\end{theorem}
\begin{proof}
We present here the structure of the proof, that is articulated into several propositions proved in the rest of the paper.
\begin{enumerate}
\item[(I.1)] This is contained in Section \ref{2503091640}: we prove that if $\hat{v}\in \Cnew(\Omega_\varepsilon(\defects))\cap \mathcal{C}^\infty(\overline{\Omega}_\varepsilon(\defects))$ is the unique solution to \eqref{2502131930}, and \eqref{eq_cofattore} holds, then  $\hat\epsilon\in \mathcal{C}^\infty(\Omega_\varepsilon(\defects);\R^{2\times2}_{\sym})$ is the unique solution to \eqref{2502131927}.

In Proposition~\ref{2408291022} we show that the minimum problem \eqref{2502131930} has a unique solution $\hat{v}\in \Cnew(\Omega_\varepsilon(\defects))$ and characterize it in terms of the Euler--Lagrange equations in weak form; the strong form is obtained in Proposition~\ref{cor_strongEL} under the assumption that $\hat{v}\in\Cnew(\Omega_\epsilon(\defects))\cap H^4(\Omega_\epsilon(\defects))$.
In Proposition~\ref{2408301408}, we show that the strain $\hat\epsilon$ defined by \eqref{eq_cofattore} 
is the unique solution to \eqref{2502131927}. The regularity of $\hat{v}$, and therefore of $\hat\epsilon$, is a consequence of the smoothness of the boundary~$\partial\Omega$.

\item[(I.2)] This is contained in Section \ref{2503091644}: we prove that if $\hat\epsilon\in \mathcal{C}^\infty(\Omega_\varepsilon(\defects);\R^{2\times2}_{\sym})$ is the unique solution to \eqref{2502131927}, and \eqref{eq_cofattore} holds, then $\hat{v}\in \Cnew(\Omega_\varepsilon(\defects))\cap \mathcal{C}^\infty(\overline{\Omega}_\varepsilon(\defects))$ is the unique solution to \eqref{2502131930}.

The strategy hinges on the superposition principle thanks to the linearity of the equations. 
From $\alpha\in\ED(\Omega)$ and $\theta\in\WD(\Omega)$, we can define $v^\ii\in \mathcal{C}^\infty(\overline{\Omega}_\varepsilon(D))$ as in \eqref{eq_vp} (the regularity is immediate to verify); from \eqref{eq_Airyoperatorgen}, we can define the corresponding stress tensor $\sigma^\ii\coloneqq \Airy(v^\ii)=\cof(\nabla^2 v^\ii)$; finally, we let $\epsilon^\ii\coloneqq \C^{-1}\sigma^\ii$.
Notice that $\epsilon^\ii,\sigma^\ii\in \mathcal{C}^\infty(\overline{\Omega}_\varepsilon(D);\R^{2\times2}_{\sym})$.
The existence, uniqueness, and regularity of the solution to \eqref{2502131927} are proved in Theorem~\ref{thm_liftabilityE}, where the %elastic
{compatible} part $\hat{\epsilon}^\cc$ is found as the unique solution to \eqref{cauchy_ee}. 
Letting $\hat v$ be the solution to \eqref{2502131930}, we define $\hat v^\cc\coloneqq \hat v-v^\ii$ and construct the corresponding $\tilde \epsilon^\cc$ via \eqref{eq_cofattore}. The proof is concluded by a uniqueness argument, once we show, in Step 3 of the proof of Theorem~\ref{thm_liftabilityE}, that $\tilde\epsilon^\cc$ also solves \eqref{cauchy_ee}.  

\item[(II)] This is contained in Proposition~\ref{prop_regularity}. \qedhere
\end{enumerate}
\end{proof}

\begin{remark}
The proof of Theorem~\ref{2502282100} bears significant mechanical implications, which we highlight here.

We find an important characterization of the unique minimizer $\hat{v}$: the Euler-Lagrange equations for problem \eqref{2502131930} in the set  $\Cnew(\Omega_\varepsilon(\defects))$ are derived in Proposition \ref{cor_strongEL} and presented in \eqref{eq_strong_EL_v_min}. 
In particular, \eqref{eq_strong_EL_v_min_bulk} represents the compatibility of the Airy potential in the open set  $\Omega_{\varepsilon}(\defects)$, in the absence of bulk defects. 
Interestingly, the boundary integrals in \eqref{eq_strong_EL_v_min_BC}, which represent incompatibility conditions formulated in terms of the Airy potential, arise as a direct consequence of the natural minimality conditions.

Furthermore, after proving a useful orthogonal decomposition of the space of competitors $\Cnew(\Omega_\varepsilon(\defects))$, we show that the minimization problem \eqref{2502131930}  can be formulated in terms of a finite-dimensional minimization involving cell formulas, which we present in the next subsection.
\end{remark}

\subsection{Orthogonal decomposition of $\Cnew(\Omega_\varepsilon(\defects))$ and cell formulas} \label{sec:251124}
In this section, we prove structural properties of the space of admissible competitors $\Cnew(\Omega_\varepsilon(\defects))$ introduced in~\eqref{eq_competitorsperforatedfloat}.
We define the bilinear, symmetric, homogeneous, and strictly positive definite form
$(\cdot, \cdot ) \colon$ $\Cnew(\Omega_\varepsilon(\defects)) \times \Cnew(\Omega_\varepsilon(\defects)) \to \mathbb{R}$ by
\begin{equation}\label{eq_formabilin}
(v_1, v_2) \mapsto \frac{1+\nu}{E} \int_{\Omega_{\varepsilon}} \big( (1-\nu)\nabla^2 v_1 : \nabla^2 v_2 - \nu [v_1, v_2] \big) \,\ud x ,
\end{equation}
where $[\cdot, \cdot]$ is the Monge-Ampère operator (see Definition~\ref{sec:24241126922}).
The form $(\cdot, \cdot )$ defines an inner product and induces the norm $\lVert v \lVert \coloneqq \sqrt{(v, v)}$ over $\Cnew(\Omega_\varepsilon(\defects))$ which is equivalent to the $H^2$ norm.
Let us consider now the orthogonal complement to $H^2_0(\Omega_{\varepsilon}(\defects))$ in $\Cnew(\Omega_\varepsilon(\defects))$ with respect to the scalar product just defined:
\begin{equation*}
    (H^2_0(\Omega_{\varepsilon}(\defects)))^{\perp} \coloneqq \big \{ v_{\perp} \in \Cnew(\Omega_\varepsilon(\defects)) : (v_0, v_{\perp}) = 0 \text{ for every $v_0 \in H_0^2(\Omega_{\varepsilon}(\defects))$} \big \}.
\end{equation*}
The generic element $v_{\perp}\in (H^2_0(\Omega_{\varepsilon}(\defects)))^{\perp}$ must have affine boundary conditions, i.e., 
\begin{equation}\label{eq_nonsichiama}
	\begin{cases}
	v_{\perp} = \partial_n v_{\perp} = 0 & \text{on $\partial \Omega$,} \\
	v_{\perp} = a_0^i + a_1^i x_1 + a_2^i x_2 & \text{on $\partial B_{\varepsilon}^i$\,, for every $i = 1, \ldots, N$,} \\
	\partial_n v_{\perp} = a_1^i n_1+a_2^i n_2 & \text{on $\partial B_{\varepsilon}^i$\,, for every $i = 1, \ldots, N$.}
	\end{cases}
\end{equation}
where $A\coloneqq (a_0^1, a_1^1, a_2^1, \ldots, a_0^N, a_1^N, a_2^N)\in\mathbb{R}^{3N}$ is the vector containing the coefficients on the~$N$ affine functions on the inner boundaries $\partial B_\varepsilon^i$ ($i=1,\ldots,N$) of $\Omega_{\varepsilon}(\defects)$.
Given $A\in\mathbb{R}^{3N}$, we will denote by $v_{\perp}^{A}$ the function in $(H^2_0(\Omega_{\varepsilon}(\defects)))^{\perp}$ that satisfies  \eqref{eq_nonsichiama}.
Moreover, for any $v_0 \in H^2_0(\Omega_{\varepsilon})$, by orthogonality
\begin{equation}
\label{eq: 250920241440}
    (v_{\perp}^{A}, v_0) = 0 \quad \Leftrightarrow \quad (1-\nu) \int_{\Omega_{\varepsilon}} \nabla^2 v_0 : \nabla^2 v_{\perp}^{A} \,\ud x - \nu \int_{\Omega_{\varepsilon}} [v_0, v_{\perp}^{A}]\,\ud x = 0.
\end{equation}
By applying Lemma~\ref{Lemma: Ciarlet's lemma} with $\xi = v_0$\,, $\eta = v_{\perp}^{A}$\,, and $\chi = 1$, the  term $\int_{\Omega_{\varepsilon}} [v_0, v_{\perp}^{A}]\,\ud x$ vanishes; therefore, the orthogonality condition~\eqref{eq: 250920241440} reduces to
\begin{equation}\label{eq: 110241726}
(v_{\perp}^A,v_0)=0 \quad\Leftrightarrow\quad \int_{\Omega_{\varepsilon}} \nabla^2 v_0 : \nabla^2 v_{\perp}^{A} \,\ud x = 0, \qquad \text{for every $v_0 \in H^2_0(\Omega_{\varepsilon}(\defects))$.}
\end{equation}
Since the second condition in~\eqref{eq: 110241726} is the weak formulation of the biharmonic equation, it follows that the generic element $v_{\perp}^{A}\in(H^2_0(\Omega_{\varepsilon}(\defects)))^{\perp}$ is the unique weak solution to the following Dirichlet problem
\begin{equation}\label{eq: 11020241240}
\begin{cases}
\Delta^2 v = 0 & \text{in $H^{-2}(\Omega_{\varepsilon}(\defects))$,} \\
v = \partial_n v = 0 & \text{on $\partial \Omega$,} \\
v = a_0^i + a_1^i x_1 + a_2^i x_2 & \text{on $\partial B_{\varepsilon}^i$ for every $i = 1, \ldots, N$,} \\
\partial_n v = a_1^i n_1+ a_2^i n_2 & \text{on $\partial B_{\varepsilon}^i$ for every $i = 1, \ldots, N$.}
\end{cases}
\end{equation} 
For any $A \in \mathbb{R}^{3N}$, the function $v_{\perp}^{A}$ belongs to $H^{m}(\Omega_{\varepsilon}(\defects))$ if $\partial \Omega \in \Ccal^{m}$ for $m \ge 4$ (see \cite[Theorem~2.20]{Gazzola09}). 
Since $\Cnew(\Omega_\varepsilon(\defects)) = H_0^2(\Omega_{\varepsilon}(\defects)) \oplus (H^2_0(\Omega_{\varepsilon}(\defects)))^{\perp} $, any $v \in \Cnew(\Omega_\varepsilon(\defects))$ can then be uniquely decomposed as 
\begin{equation}\label{eq:110241530}
v = v_0 + v_{\perp}^{A}\,,
\end{equation}
where $v_{\perp}^{A}$ which is the unique solution to problem \eqref{eq: 11020241240} (and is determined by the values of $v$ on $\partial\Omega_\varepsilon(\defects)$), and $v_0\in H^2_0(\Omega_{\varepsilon}(\defects))$ is determined by orthogonality.

Owing to the linearity of problem \eqref{eq: 11020241240}, we may deduce that $v_{\perp}^{A}$ depends linearly on the coefficients $A \in \mathbb{R}^{3N}$.
Then it can be expressed as
\begin{equation}\label{260920241218}
v^{A}_{\perp}(x) =  \sum\limits_{i=1}^N \sum\limits_{r=0}^2 a_r^i \, \kappa_r^i(x) =\langle A,\kappa(x)\rangle,\qquad\text{for $x\in\Omega_\varepsilon(\defects)$,}
\end{equation}
where $\kappa \coloneqq (\kappa^1_0, \kappa^1_1, \kappa^1_2, \ldots, \kappa^N_0, \kappa^N_1, \kappa^N_2)\colon\Omega_\varepsilon(\defects)\to\mathbb{R}^{3N}$ is the vector whose elements are obtained from the following cell %formula 
problems: for every $i=1,\ldots,N$, the functions $\kappa^i_0$ and $\kappa^i_r$ ($r=1,2$) are the unique solutions to, respectively,
\begin{equation}\label{eq:2410171611}
\begin{cases}
\Delta^2 \kappa_0 = 0 & \text{in $\Omega_{\varepsilon}(\defects)$,} \\
\kappa_0 = \partial_n \kappa_0 = 0 & \text{on $\partial\Omega$,} \\
\kappa_0 = 1 & \text{on $\partial B_{\varepsilon}^i$\,,} \\
\partial_n \kappa_0 = 0 & \text{on $\partial B_{\varepsilon}^i$\,,} \\
\kappa_0 = \partial_n \kappa_0 = 0 & \text{on $\partial B_{\varepsilon}^I$, if $I \ne i$,} 
\end{cases}
\quad\text{and}\quad
\begin{cases}
\Delta^2 \kappa_r = 0 & \text{in $\Omega_{\varepsilon}(\defects)$,} \\
\kappa_r = \partial_n \kappa_r = 0 & \text{on $\partial\Omega$,} \\
\kappa_r = x_r & \text{on $\partial B_{\varepsilon}^i$\,,} \\
\partial_n \kappa_r = n_r & \text{on $\partial B_{\varepsilon}^i$\,,} \\
\kappa_r = \partial_n \kappa_r = 0 & \text{on $\partial B_{\varepsilon}^I$, if $I \ne i$.}
\end{cases}	
\end{equation}
These cell formulas depend on the geometry and topology of the domain but not
on the mechanical parameters of the problem.
One advantage of these cell formulas is that they  simplify the treatment of the boundary conditions for the stress: by ensuring homogeneous Neumann conditions, they transform the minimization problem in $\Cnew(\Omega_\varepsilon(\defects))$ into a finite-dimensional algebraic minimization, see also \cite{CARBONARA7,PETROLO20042471}.
In particular, $\kappa_0^i$ determines the constant part of the solution $v_\perp^A$ at the boundary $\partial B_\varepsilon^i$\,; similarly, $\kappa_1^i$ and $\kappa_2^i$ determine the $x_1$- and $x_2$ contributions of $v_\perp^A$ at the boundary $\partial B_\varepsilon^i$\,.
We would like to stress that formulas \eqref{eq:2410171611} serve as a preliminary step toward studying the homogenization of a system of regularized dislocations and disclinations as the core radius vanishes.

As a consequence of \eqref{eq:2410171611}, we have the following remark.
\begin{remark}\label{rem_reg}
Regularity results for higher order Dirichlet problems (see again \cite[Theorem~2.20]{Gazzola09}) ensure that for every $i=1,\ldots,N$ and every $r=0,1,2$ the function $\kappa_r^i \in H^{m}(\Omega_{\varepsilon}(\defects))$ if $\partial \Omega \in \Ccal^{m}$ for $m \ge 4$.
By \eqref{260920241218}, the same regularity is inherited by $v_\perp^A$\,.
\end{remark}
We conclude by computing two crucial ingredients for Proposition~\ref{2408291024} below: the norm of $\nabla^2 v_{\perp}^{A}$ and the matrix $\mathcal{M} \in \mathbb{R}^{3N\times3N}_{\sym}$. 
Letting $\mathcal{K}_r^i \coloneqq \nabla^2 \kappa_r^i\colon\Omega_\varepsilon(\defects)\to\mathbb{R}^{2\times2}_{\sym}$, from \eqref{260920241218} we have that
\begin{equation}\label{11020241610}
\nabla^2 v^{A}_{\perp}(x) = \sum\limits_{i=1}^N \sum\limits_{r=0}^2 a_r^i \hspace{0.2em} \mathcal{K}_r^i(x),
\;\; 
\text{for $x\in\Omega_\varepsilon(\defects)$,}
\;\;
\text{and}
\;\;
\lVert \nabla^2 v_{\perp}^{A} \lVert_{L^2(\Omega_{\varepsilon}, \mathbb{R}^{2\times2})}^2 = \langle\mathcal{M} {A},A\rangle,
\end{equation}
where $\mathcal{M}$ is the $(3N)\times(3N)$ block matrix, whose blocks $\mathcal{M}_{ij}$ ($i,j=1,\ldots,N$) are the $3\times3$ matrices defined as
\begin{equation}\label{eq:202411261011}
\big(\mathcal{M}_{ij}\big)_{rs} \coloneqq \int_{\Omega_\varepsilon} \mathcal{K}^i_r(x) : \mathcal{K}^j_s(x)\,\ud x,\qquad \text{for $r,s=0,1,2$;}
\end{equation}
notice that $\mathcal{M}$ is symmetric and, thanks to~\eqref{11020241610}, it is positive definite.

\subsection{Mechanical equilibrium in the Airy potential formulation}\label{2503091640}
\begin{proposition}\label{2408291022}
Let $\Omega\subset\R^2$ be of class $\Ccal^2$ and as in \eqref{eq_Omega}.
Let $\alpha\in\ED(\Omega)$ and $\theta\in\WD(\Omega)$ satisfy \eqref{eq_disj_supp}, let $\defects$ be as in \eqref{eq_defects}, and let $\varepsilon\in(0,\varepsilon_0)$.
Then the minimum problem \eqref{2502131930} has a unique minimizer $\hat{v}$, which is characterized by
\begin{equation}\label{2408291139}
\!\!\!
0=\frac{1+\nu}{E}
\int_{\Omega_{\varepsilon}}\!\!\! \big[\nabla^2 \hat{v}:\nabla^2\phi-\nu\Delta \hat{v}\,\Delta\phi\big]\,\ud x +\sum_{j=1}^J \ave_{\partial B_{\ce}^j} \!\langle\nabla \phi, \Pi(b^j)\rangle\,\ud\Huno +\sum_{k=1}^K \ave_{\partial B_\ce^k} \! s^k\phi\,\ud\Huno\,
\end{equation}
for every $\phi\in\Cnew(\Omega_\varepsilon(\defects))$.
\end{proposition}
\begin{proof}
We start by applying the direct method of the calculus of variations to prove the existence of a minimizer for $\I$ in $\Cnew(\Omega_\varepsilon(\defects))$.
Let $\{v_m\}\subset \Cnew(\Omega_\varepsilon(\defects)$ be a minimizing sequence; by the coercivity of $\G$ and thanks to the Friedrich's inequality, there exists $M>0$ such that $\lVert v_m \rVert^2_{H^2(\Omega)}\leq M$.
Consequently, owing to the fact that $\Cnew(\Omega_\varepsilon(\defects))$ is weakly closed with respect to the $H^2$ convergence, they converge, up to the extraction of a subsequence to a function $\hat{v}\in\Cnew(\Omega_\varepsilon(\defects))$ which is a minimizer of $\I$. 
This minimizer is unique, thanks to the strict convexity of $\I$.
To prove the characterization \eqref{2408291139}, it suffices to impose that the first variation of~$\I$ vanish.
\end{proof}

Relying on the orthogonal decomposition from Section~\ref{sec:251124} and on the structure \eqref{260920241218} of $v_\perp^A$, we now prove that the minimization problem \eqref{2502131930} is equivalent to minimizing the energy $\I(\cdot,\Omega_\varepsilon)$ over a $3N$-dimensional class of competitors. 
\begin{proposition}\label{2408291024}
Let $\Omega\subset\R^2$ be as in \eqref{eq_Omega} and let us assume that $\partial\Omega$ is of class $\Ccal^4$.
Let $\alpha\in\ED(\Omega)$ and $\theta\in\WD(\Omega)$ satisfy \eqref{eq_disj_supp}, let $\defects$ be as in \eqref{eq_defects}, and let $\varepsilon\in(0,\varepsilon_0)$.
Then the unique solution $\hat{v}$ to the minimum problem \eqref{2502131930} is given by
\begin{equation}\label{eq:202411261130}
\hat{v}(x) = - \frac{E}{1-\nu^2} \langle \mathcal{M}^{-1} \kappa(x),\Phi\rangle, \qquad\text{for $x\in\Omega_\epsilon(\defects)$,}
\end{equation}
where $\Phi\coloneqq (s^1, b_2^1,- b_1^1, s^2, b_2^2, -b_1^2, \ldots, s^N, b_2^N,- b_1^N)$.
Moreover, $\hat v$ inherits the regularity of the solutions to problems \eqref{eq:2410171611}.
\end{proposition}
\begin{proof}
Let us consider $v \in \Cnew(\Omega_\varepsilon(\defects))$; %and substitute Eqn. 
by plugging the additive decomposition~\eqref{eq:110241530} into the energy in \eqref{eq_energyperforated}, we get
\begin{equation}\label{eq:202411261041}
\begin{split}
\I(v;\Omega_{\varepsilon}) =&\, \I(v_0 + v_{\perp}^{A};\Omega_{\varepsilon}) \\
=&\, \G(v_0 + v_{\perp}^{A};\Omega_{\varepsilon}) + \sum_{j=1}^J\frac{1}{2\pi\ce}\int_{\partial B_{\ce}^j} \langle\nabla v_{\perp}^{A}, \Pi(b^j)\rangle\,\ud\Huno
+\sum_{k=1}^K \frac{s^k}{2\pi\ce}\int_{\partial B_\ce^k} v_{\perp}^{A} \,\ud\Huno.
\end{split}
\end{equation}
By the identity $[v,v] = (\Delta v)^2 - |\nabla^2 v|^2 $, holding true a.e.~in $\Omega_\varepsilon(\defects)$, and recalling~\eqref{eq_formabilin}, we have that 
$\mathcal{G}(v;\Omega) = \frac12 \lVert v \lVert^2$; then
\begin{equation}\label{eq_elenG}
\G(v_0 + v_{\perp}^{A};\Omega_{\varepsilon}) = \G(v_0;\Omega_{\varepsilon}) + (v_0, v_{\perp}^{A}) + \G(v_{\perp}^{A};\Omega_{\varepsilon}) = \G(v_0;\Omega_{\varepsilon})+ \G(v_{\perp}^{A};\Omega_{\varepsilon}), 
\end{equation}
by orthogonality. 
Invoking that $\partial \Omega \in \Ccal^4$, we obtain that $v_{\perp}^{A} \in H^4(\Omega_{\varepsilon}(\defects))$ (see Remark~\ref{rem_reg}) and we may employ Lemma~\ref{Lemma:2410171413} with $\xi= \eta=v_{\perp}^{A}$ and $\chi=1$ to conclude that
\begin{equation}
\int_{\Omega_\varepsilon} [v_{\perp}^{A}, v_{\perp}^{A}] \,\ud x = 0, \qquad \text{for every $A \in \mathbb{R}^{3N}$,}
\end{equation}
so that, by~\eqref{11020241610}, the elastic energy in \eqref{eq_elenG} reduces to
\begin{equation}\label{eq:202411261051}
\G(v_0 + v_{\perp}^{A};\Omega_{\varepsilon}) = \G(v_0;\Omega_{\varepsilon}) + \frac{1-\nu^2}{2E} \langle\mathcal{M} {A},A\rangle.
\end{equation}
We now turn our attention to the boundary integrals in \eqref{eq:202411261041}. 
Since membership in $\Cnew(\Omega_\varepsilon(\defects))$ means having an affine trace on $\partial\Omega_\varepsilon(\defects)$, we have
\begin{equation} \label{eq:202411261055}
\begin{split}
\sum_{j=1}^J\frac{1}{2\pi\ce}\int_{\partial B_{\ce}^j} \langle\nabla v, \Pi(b^j)\rangle\,\ud\Huno =&\, \sum_{j=1}^J \langle ({a}^j_1, {a}^j_2)^\top, \Pi(b^j)\rangle = \langle \Phi_D , A \rangle,\\ 
\sum_{k=1}^K \frac{s^k}{2\pi\ce}\int_{\partial B_\ce^k} v\,\ud\Huno = &\, \sum_{k=1}^K s^k a^k_0 = \langle \Phi_d , A \rangle,
\end{split}
\end{equation}
where we set, recalling~\eqref{eq_sb_general},
\begin{equation*}% \label{eq:2410181027}
\begin{split}
\Phi_D \coloneqq &\, (0, b_2^1,- b_1^1, 0, b_2^2, b_1^2, \hdots, 0, b_2^N,- b_1^N)^\top \in \mathbb{R}^{3N}\\
\Phi_d \coloneqq &\, (s^1, 0, 0, s^2, 0, 0, \hdots, s^N, 0, 0)^\top \in \mathbb{R}^{3N}.
\end{split}
\end{equation*}
Combining \eqref{eq:202411261051} and \eqref{eq:202411261055} and defining $\Phi \coloneqq \Phi_D + \Phi_d$, the expression in \eqref{eq:202411261041} becomes
\begin{equation}\label{eq_611}
\begin{split}
\I(v;\Omega_{\varepsilon}) =
\G(v_0;\Omega_{\varepsilon}) + \frac{1-\nu^2}{2E} \langle \mathcal{M} A, A \rangle + \langle \Phi, A \rangle.
\end{split}
\end{equation}
Denoting by $\hat{v}$ the unique minimizer of $\I$ in $\Cnew(\Omega_\varepsilon(\defects))$ (see Proposition \ref{2408291022}), owing to~\eqref{eq: 11020241240}, \eqref{eq:110241530}, and~\eqref{eq_611}, we have
\begin{equation*}
\begin{split}
\I(\hat{v};\Omega_{\varepsilon})= &\, \min \big\{ \I(v;\Omega_{\varepsilon}) : v \in \Cnew(\Omega_\varepsilon(\defects)) \big\} \\
=&\, \min \big\{ \I(v_0 + v^{A}_{\perp};\Omega_{\varepsilon}) : v_0 \in H^2_0(\Omega_\varepsilon(\defects)), A \in \mathbb{R}^{3N} \big\} \\
=&\, \min \bigg\{ \G(v_0;\Omega_{\varepsilon}) + \frac{1-\nu^2}{2E} \langle \mathcal{M} A, A\rangle + \langle \Phi, A \rangle : v_0 \in H^2_0(\Omega_\varepsilon(\defects)), A \in \mathbb{R}^{3N} \bigg\} \\
=&\, \min \big\{ \G(v_0;\Omega_{\varepsilon}) : v_0 \in H^2_0(\Omega_{\varepsilon})\big\} + \min \bigg\{ \frac{1-\nu^2}{2E} \langle \mathcal{M} A, A\rangle + \langle \Phi, A \rangle : {A} \in \mathbb{R}^{3N} \bigg\}\\
=&\, \min \bigg\{ \frac{1-\nu^2}{2E} \langle \mathcal{M} A, A\rangle + \langle \Phi, A \rangle : {A} \in \mathbb{R}^{3N} \bigg\}.
\end{split}
\end{equation*}
Indeed, since we noticed that $\G(\cdot;\Omega_{\varepsilon})$ is a norm in $H^2_0(\Omega_\varepsilon)$, then $\min \{ \G(v_0;\Omega_{\varepsilon}) : v_0 \in H^2_0(\Omega_{\varepsilon})\} = 0$ and $v_0 \equiv 0$. Furthermore, it can be easily verified that
\begin{equation}
\min_{{A} \in \mathbb{R}^3} \bigg\{ \frac{1-\nu^2}{2E} \langle \mathcal{M} A, A\rangle + \langle \Phi, A \rangle  \bigg\} = -\frac{E}{2(1-\nu^2)} \langle \mathcal{M}^{-1} \Phi, \Phi \rangle
\end{equation}
which is attained at 
\begin{equation}
\widehat A \coloneqq \argmin_{A \in \mathbb{R}^{3N}} \bigg\{ \frac{1-\nu^2}{2E} \langle \mathcal{M} A, A\rangle + \langle \Phi, A \rangle \bigg\} = -\frac{E}{1-\nu^2} \mathcal{M}^{-1} \Phi.
\end{equation}
Therefore, recalling \eqref{260920241218}, we have that $\hat v= v_{\perp}^{\widehat{A}} = -\frac{E}{1-\nu^2}\langle\mathcal{M}^{-1} \Phi, \kappa \rangle$,
which is \eqref{eq:202411261130}. 

The regularity of $\hat{v}$ is a consequence of its linear structure and Remark~\ref{rem_reg}.
\end{proof}

We now give  the strong form of \eqref{2408291139}, relying on the improved regularity of $\hat{v}$ from Proposition~\ref{2408291024}.
\begin{proposition}
\label{cor_strongEL}
Let $\Omega\subset\R^2$ be as in \eqref{eq_Omega}, 
let $\alpha\in\ED(\Omega)$ and $\theta\in\WD(\Omega)$ be such that \eqref{eq_disj_supp} holds true, and let $\defects$ be as in \eqref{eq_defects}. 
If the solution $\hat{v}$ to the minimum problem \eqref{2502131930} belongs to $H^4(\Omega_\varepsilon(\defects))\cap\Cnew(\Omega_\varepsilon(\defects))$, then the conditions \eqref{2408291139} read
\begin{subequations}\label{eq_strong_EL_v_min}
\begin{equation}\label{eq_strong_EL_v_min_bulk}
\frac{1-\nu^2}{E}\Delta^2 \hat{v}=0 \qquad\text{in $L^2(\Omega_\varepsilon(\defects))$}
\end{equation}
and 
\begin{equation}\label{eq_strong_EL_v_min_BC}
\begin{cases}
\displaystyle \frac{1-\nu^2}{E} \int_{\partial B_\varepsilon^i} \partial_n(\Delta \hat{v})\,\ud\Huno=\tilde\theta(\Omega^i)\,,\\[3mm]
\displaystyle \frac{1-\nu^2}{E} \int_{\partial B_\varepsilon^i} \bigg(x_1\partial_t(\Delta \hat{v})-x_2\partial_n(\Delta \hat{v})+\frac{(\nabla^2 \hat{v}\,t)_1}{1-\nu}\bigg)\,\ud\Huno=(\tilde\alpha(\Omega^i))_1\,,\\[3mm]
\displaystyle \frac{1-\nu^2}{E} \int_{\partial B_\varepsilon^i} \bigg(x_1\partial_n(\Delta \hat{v})+x_2\partial_t(\Delta \hat{v})+\frac{(\nabla^2 \hat{v}\,t)_2}{1-\nu}\bigg)\,\ud\Huno=(\tilde\alpha(\Omega^i))_2\,,
\end{cases}
 \end{equation}
for every $i=1,\ldots,N$, where $n$ is the outer unit normal to $\Omega_\varepsilon(\defects)$, the boundaries are oriented clockwise, and $\tilde\alpha$ and $\tilde\theta$ are defined in \eqref{eq_extended_measures}.
\end{subequations}
\end{proposition}
\begin{proof}
Equation \eqref{2408291139}, $\hat{v}\in H^4(\Omega_\varepsilon(\defects))$,  
$\phi\in \Cnew(\Omega_{\varepsilon}(\defects))$, and integration by parts (the details of which can be found in Lemma~\ref{lemma_ibp}) imply that 
\begin{equation*}%\label{2408282214}
\begin{aligned}
0=&\, \frac{1-\nu^2}{E} \int_{\Omega_{\varepsilon}} \Delta^2 \hat{v}\,\phi\,\ud x +\frac{1-\nu^2}{E} \sum_{i=1}^N \int_{\partial B_{\varepsilon}^i}  \bigg(\Delta \hat{v}\, \partial_n\phi-\phi\partial_n(\Delta \hat{v})+\frac{\langle \nabla^2\hat{v}\,t,\Pi(\nabla \phi)\rangle}{1-\nu}\bigg)\,\ud\Huno\\ 
&\, +\sum_{j=1}^J\frac{1}{2\pi\ce}\int_{\partial B_{\ce}^j} \langle \nabla \phi, \Pi(b^j)\rangle\,\ud\Huno +\sum_{k=1}^K \frac{s^k}{2\pi\ce}\int_{\partial B_\ce^k} \phi\,\ud\Huno\,.
\end{aligned}
\end{equation*}
We can now apply the fundamental lemma of the calculus of variation to obtain \eqref{eq_strong_EL_v_min_bulk}. 
Moreover, recalling that any $\phi\in\Cnew(\Omega_\varepsilon(\defects))$ is such that $\phi|_{\partial B_\varepsilon^i}=a^i|_{\partial B_\varepsilon^i}$ and 
and $\partial_n\phi=\nabla a^i\cdot n$ on $\partial B_\ce^i$ for some affine functions $a^i$, for every $i=1,\ldots,N$, we can choose specific test functions that are zero on all but one of these boundaries, so that the boundary terms above yield 
\begin{equation}\label{eq_integratedEL_reg_BC}
\begin{split}
0=&\, \frac{1-\nu^2}{E}\int_{\partial B_\varepsilon^i} \bigg(\Delta \hat{v}\,\partial_n\phi-\phi\partial_n(\Delta \hat{v})+\frac{\langle \nabla^2\hat{v}\,t,\Pi(\nabla\phi) \rangle}{1-\nu}\bigg)\,\ud\Huno \\
&\, + \frac{1}{2\pi\ce}\int_{\partial B_{\ce}^i} \langle  \nabla \phi, \Pi(b^i)\rangle\,\ud\Huno +\frac{s^i}{2\pi\ce}\int_{\partial B_\ce^i} \phi\,\ud\Huno\,,
\end{split}
\end{equation}
for all $i=1,\dots,N$, and for all $\phi\in\Cnew(\Omega_\varepsilon(\defects))$. 
By choosing $\phi(x_1,x_2)=a_0$ and plugging it in \eqref{eq_integratedEL_reg_BC}, we obtain
$$0=-\frac{1-\nu^2}{E} \int_{\partial B_\varepsilon^i} a_0\partial_n(\Delta \hat{v})\,\ud\Huno + \frac{s^i}{2\pi\varepsilon}\int_{\partial B_\varepsilon^i} a_0\,\ud\Huno\,, $$ 
which gives the first equation in \eqref{eq_strong_EL_v_min_BC}.
By choosing $\phi(x_1,x_2)=a_2x_2$ and plugging it in \eqref{eq_integratedEL_reg_BC}, we obtain
\[\begin{split}
0=&\, \frac{1-\nu^2}{E} \int_{\partial B_\varepsilon^i} a_2 \bigg(\Delta\hat{v}\,n_2-x_2\partial_n(\Delta\hat{v})+\frac{(\nabla^2\hat{v}\,t)_1}{1-\nu}\bigg)\ud\Huno-\frac{1}{2\pi\varepsilon} \int_{\partial B_\varepsilon^i} a_2b_1^i\,\ud\Huno \\
=&\, a_2\frac{1-\nu^2}{E} \int_{\partial B_\varepsilon^i} \bigg( x_1\partial_t(\Delta\hat{v})- x_2\partial_n(\Delta\hat{v})+ \frac{(\nabla^2\hat{v}\,t)_1}{1-\nu}\bigg) \ud\Huno - a_2b_1^i\,,
\end{split}\]
where we have used that $n_2=-\ud x_1$ (see~\eqref{202308152252}) and~\eqref{202308142315} with $r=1$. 
Rearranging the terms, this gives the second equation in~\eqref{eq_strong_EL_v_min_BC}.
Finally, by choosing $\phi(x_1,x_2)=a_1x_1$ and plugging it in \eqref{eq_integratedEL_reg_BC}, and by using that $n_1=\ud x_2$ (see~\eqref{202308152252}) and~\eqref{202308142315} with $r=2$, we obtain the third equation in~\eqref{eq_strong_EL_v_min_BC}.
The proposition is proved.
\end{proof}

In the following proposition we translate the results of Proposition~\ref{cor_strongEL} in the stress-strain formulation.
\begin{proposition}\label{2408301408}
Let $\Omega$ be as in \eqref{eq_Omega} and let us assume that $\partial \Omega\in \mathcal{C}^\infty$. 
Let $\alpha\in\ED(\Omega)$ and $\theta\in\WD(\Omega)$ satisfy \eqref{eq_disj_supp}, let $v^\ii$ be defined as in \eqref{eq_vp} and let $\defects$ be as in \eqref{eq_defects}.
Let $\hat{v}\in\Cnew(\Omega_\varepsilon(\defects))$ be the unique solution to the minimum problem in \eqref{2502131930}.
Let $\hat{\sigma}=\Airy(\hat{v})=\cof(\nabla^2 \hat{v})$, as in \eqref{eq_Airyoperatorgen}, and let $\hat{\epsilon}\coloneqq\C^{-1}\hat{\sigma}$.
Then the following conditions are satisfied 
\begin{equation}\label{2408291029}
\begin{cases}
\Dive\hat{\sigma} = 0 & \text{in $\Omega_\varepsilon$,}\\
\hat{\sigma}n = 0 & \text{on $\partial\Omega_\varepsilon$} \\[2mm]
\curl\Curl\hat{\epsilon} = 0 & \text{in $\Omega_\varepsilon$,} \\
\displaystyle \int_{\partial B_\varepsilon^i} \sum_{q=1}^2 (\hat\epsilon_{1q,2}-\hat\epsilon_{q2,1})\,\ud x_q = \tilde\theta(\Omega^i) & \text{for $i = 1,\ldots, N$,} \\[4mm]
\displaystyle \int_{\partial B_\varepsilon^i} \sum_{c,q=1}^2 [(\hat\epsilon_{rc}-x_q(\hat\epsilon_{rc,q}-\hat\epsilon_{cq,r})]\,\ud x_c = \big(\tilde\alpha(\Omega^i)\big)_r & \text{for $i = 1,\ldots, N$, and $r=1,2$,}
\end{cases}
\end{equation}
where the boundaries are oriented counter-clockwise, and $\hat\alpha$ and $\hat\theta$ are defined in \eqref{eq_extended_measures}. 
\end{proposition}
\begin{proof}
Thanks to the regularity of the boundary $\partial\Omega_\varepsilon(\defects)$, by Proposition~\ref{2408291024}, $\hat{v}\in \mathcal{C}^\infty(\overline{\Omega}_\varepsilon(\defects))$. 
The first condition in \eqref{2408291029} is immediate from the definition of $\hat{\sigma}$ via the Airy operator \eqref{eq_Airyoperator}, whereas the third condition comes from \eqref{eq_strong_EL_v_min_bulk} thanks to the first identity in \eqref{eq_formaltrans}.
The second identity in \eqref{eq_formaltrans}, the smoothness of the boundary $\partial\Omega_\varepsilon$, and \eqref{eq_hesstang0} put us in a position to apply \cite[Proposition~A.2]{CDLM2024} which implies, by \eqref{eq_202409051159}, the second condition in \eqref{2408291029}.
Finally, conditions~\eqref{eq_strong_EL_v_min_BC} and Proposition~\ref{prop_202409051553} yield the last two conditions in~\eqref{2408291029}.
\end{proof}

As a direct consequence of the previous proposition, we are now in a position to clarify the {role of the compatible} %of the elastic 
part of the strain, upon the introduction of an additive %elasto-plastic
{compatible-incompatible} decomposition.
\begin{proposition}\label{prop_regularity}
Let $\Omega$ be as in \eqref{eq_Omega} and let us assume that $\partial \Omega\in \mathcal{C}^\infty$. 
Let $\alpha\in\ED(\Omega)$ and $\theta\in\WD(\Omega)$ satisfy \eqref{eq_disj_supp}, let $v^\ii$ be defined as in \eqref{eq_vp} and let $\defects$ be as in \eqref{eq_defects}.
Let $\hat{v}\in\Cnew(\Omega_\varepsilon(\defects))$ be the unique solution to the minimum problem in \eqref{2502131930}.
Then there exists a unique (up to rigid motions) function $\hat{u}^\cc\in \mathcal{C}^\infty(\overline{\Omega}_\varepsilon(\defects);\R^{2})$ such that $\hat{\epsilon}^\cc=\nabla^\sym \hat{u}^\cc$, where 
\begin{equation}\label{eq_epsilon_from_sigma}
\hat{\epsilon}^\cc\coloneqq \C^{-1}[\Airy(\hat{v})-\sigma^\ii]=\C^{-1}[\cof(\nabla^2 \hat{v})-\sigma^\ii],
\end{equation}
where $\sigma^\ii=\Airy(v^\ii)=\cof(\nabla^2 v^\ii)$, as in \eqref{eq_Airyoperatorgen}.
\end{proposition}
\begin{proof}
Thanks to the regularity of the boundary $\partial\Omega_\varepsilon(\defects)$, by Proposition~\ref{2408291024}, $\hat{v}\in \mathcal{C}^\infty(\overline{\Omega}_\varepsilon(\defects))$. 
The function
$\hat{v}^\cc \coloneqq \hat{v}-v^\ii\in \mathcal{C}^\infty(\overline{\Omega}_\varepsilon(\defects))$
satisfies the following conditions
\begin{equation*}%\label{eq_v^e_min_bulk}
\frac{1-\nu^2}{E}\Delta^2 \hat{v}^\cc=0\qquad \text{in $\Omega_\varepsilon(\defects)$,}
\end{equation*}
which follows from \eqref{eq_strong_EL_v_min_bulk} and \eqref{eq_necessary_incompatible_bulk}, and, for every $i=1,\ldots,N$,
\begin{equation*}%\label{eq_v^e_min_BC1}
\begin{cases}
\displaystyle \frac{1-\nu^2}{E} \int_{\partial B_\varepsilon^i} \partial_n(\Delta \hat{v}^\cc)\,\ud\Huno=0\,,\\[3mm]
\displaystyle \frac{1-\nu^2}{E} \int_{\partial B_\varepsilon^i} \bigg(x_1\partial_t(\Delta \hat{v}^\cc)-x_2\partial_n(\Delta \hat{v}^\cc)+\frac{(\nabla^2 \hat{v}^\cc\,t)_1}{1-\nu}\bigg)\,\ud\Huno=0\,,\\[3mm]
\displaystyle \frac{1-\nu^2}{E} \int_{\partial B_\varepsilon^i} \bigg(x_1\partial_n(\Delta \hat{v}^\cc)+x_2\partial_t(\Delta \hat{v}^\cc)+\frac{(\nabla^2 \hat{v}^\cc\,t)_2}{1-\nu}\bigg)\,\ud\Huno=0\,,
\end{cases}
\end{equation*}
which follow from \eqref{eq_strong_EL_v_min_BC} and \eqref{eq_necessary_incompatible_boundary}. 
Notice that $\hat{v}^\cc$ is subject to the boundary condition $\nabla^2 \hat{v}^\cc\,t=-\nabla^2 v^\ii\,t$ on $\partial\Omega_\varepsilon(\defects)$.
Let us define $\hat{\sigma}^\cc\coloneqq\Airy(\hat{v}^\cc)=\cof(\nabla^2 \hat{v}^\cc)$, so that $\hat{\epsilon}^\cc\coloneqq \C^{-1}\hat{\sigma}^\cc$ can be represented as in \eqref{eq_epsilon_from_sigma}.
By \cite[Theorem~8.8]{Selvadurai} there exists $\hat{u}^\cc\in \mathcal{C}^\infty(\overline{\Omega}_\varepsilon(\defects);\R^{2})$ (the regularity follows from that of $\hat{v}^\cc$) such that $\hat{\epsilon}^\cc=\nabla^\sym\hat{u}^\cc$, and this concludes the proof. 
\end{proof} 

\subsection{Mechanical equilibrium in the stress-strain formulation}\label{2503091644}
We start by recalling a theorem by T.~W.~Ting on a weak formulation of the Saint-Venant principle (see Proposition~\ref{prop_weakSV}) in three dimensions.
We use this to show the existence of solutions to the mechanical equilibrium problem in the stress-strain formulation.
We do so by adapting Ting's Theorem to our planar setting.
\begin{theorem}[{Ting \cite{Ting1974}; see \cite[Theorem~3]{GK2005} for this formulation}]\label{thm_Ting}
Let $Y\subset\mathbb{R}^3$ be a bounded, connected, open set with Lipschitz-continuous boundary $\partial Y$.
Let $E\in L^2(Y;\mathbb{R}^{3\times3}_{\sym})$ be such that 
\begin{equation}\label{eq_orthogonalityE}
\int_{Y} E(x):S(x)\,\ud x=0\qquad\text{for every $S\in \Sigma_{\ad}$\,,}
\end{equation}
where $\Sigma_{\ad}$ is the closure in $L^2(Y;\mathbb{R}^{3\times3}_{\sym})$ of the linear space
\begin{equation}\label{eq_V}
\mathcal{V}\coloneqq\big\{ S\in\mathcal{D}(Y;\R^{3\times3}_{\sym}): \Dive S=0\text{ in }Y\big\}.
\end{equation}
Then there exists a vector field $U\in H^1(Y;\mathbb{R}^3)$ that satisfies $E=\nabla^{\sym} U$.
\end{theorem}
We remark that the domain $Y$ need not be simply connected for the result to hold true. 
For sets $Y\subset\mathbb{R}^3$ as in the hypothesis of Theorem~\ref{thm_Ting} and $\Omega\subset\R^{2}$ bounded, connected, and open with Lipschitz boundary, we define the spaces
\[\begin{split}
K(Y)\coloneqq &\, \{E\in L^2(Y;\mathbb{R}^{3\times3}_{\sym}):\text{\eqref{eq_orthogonalityE} holds true}\},\\
K'(\Omega)\coloneqq &\, \{e\in L^2(\Omega;\R^{2\times2}_{\sym}):\underline{e}\in K(\Omega\times(0,1))\},
\end{split}\]
and notice that (i) the constraint \eqref{eq_orthogonalityE} is weakly closed in $L^2(Y;\R^{3\times 3}_{\sym})$, so that $K(Y)$ is a weakly closed subspace of $L^2(Y;\R^{3\times 3}_{\sym})$; (ii) $K'(\Omega)$ is a linear subspace of $K(\Omega\times(0,1))$, so that $K'(\Omega)$ is a weakly closed subspace of $L^2(\Omega\times(0,1);\R^{3\times 3}_{\sym})$.

\begin{theorem}\label{thm_liftabilityE}
Let $\Omega$ be as in \eqref{eq_Omega} and let us assume that $\partial \Omega\in\mathcal{C}^\infty$.
Let $\alpha\in\ED(\Omega)$ and $\theta\in\WD(\Omega)$ satisfy \eqref{eq_disj_supp}, let $v^\ii$ be defined as in \eqref{eq_vp} and let $\defects$ be as in \eqref{eq_defects}.
Let $\sigma^\ii\coloneqq \Airy(v^\ii)=\cof(\nabla^2 v^\ii)$, as in \eqref{eq_Airyoperatorgen}, and let $\epsilon^\ii\coloneqq \C^{-1}\sigma^\ii$.
For every $\epsilon\in L^2(\Omega_\epsilon(\defects);\mathbb{R}^{2\times2}_{\sym})$, let us define the energy functional
\begin{equation}\label{eq_E}
\mathcal{E}(\epsilon)\coloneqq \frac12\int_{\Omega_\varepsilon(\defects)} (\C\epsilon(x)+\sigma^\ii(x)):(\epsilon(x)+\epsilon^\ii(x))\,\ud x.
\end{equation}
Then there exists a unique solution $\epsilon^\cc$ to the minimum problem
\begin{equation}\label{eq_min_3d}
\min\{\mathcal{E}(\epsilon): \epsilon\in K'(\Omega_\varepsilon)\},
\end{equation}
which satisfies the following necessary conditions 
\begin{equation}\label{cauchy_ee}
\begin{cases}
\Dive\C\epsilon^\cc=0&\text{in $\Omega_{\varepsilon}$\,,}\\
\C\epsilon^\cc\,n=-\C\epsilon^\ii\,n&\textrm{on $\partial\Omega_{\varepsilon}$\,,}\\[2mm]
\curl\Curl \epsilon^\cc=0&\text{in $\Omega_{\varepsilon}$\,,}\\
\displaystyle \int_{\partial B_\varepsilon^i} \big[(\epsilon^\cc)_{rq,c}-(\epsilon^\cc)_{qc,r}\big]\,\ud x_q = 0  & \text{for $i = 1,\ldots, N$} \\[3mm]
\displaystyle \int_{\partial B_\varepsilon^i} \big[(\epsilon^\cc)_{rc}-x_q\big((\epsilon^\cc)_{rc,q}-(\epsilon^\cc)_{cq,r}\big)\big]\,\ud x_c = 0 & \text{for $i = 1,\ldots, N$ and $r=1,2$,}
\end{cases}
\end{equation}
where the boundaries are oriented counter-clockwise.
Moreover, $\underline{\epsilon}^\cc$ is the symmetrized gradient of a field $u^\cc\in H^1(\Omega_\varepsilon;\R^{3})$ of the form
\begin{equation}\label{eq_explicit_ue}
u^\cc(x_1,x_2,x_3)=
\begin{pmatrix}
u_1^\cc(x_1,x_2)-c_{13}x_3\\
u_2^\cc(x_1,x_2)-c_{23}x_3\\
c_{13}x_1+c_{23}x_2
\end{pmatrix};
\end{equation}
Finally, the field $\epsilon\coloneqq\epsilon^\cc+\epsilon^\ii$ satisfies the following conditions 
\begin{equation}\label{eq_sistema6.7}
\begin{cases}
\Dive\C\epsilon = 0 & \text{in $\Omega_{\varepsilon}$\,,}\\			
\C\epsilon n = 0 & \text{on $\partial\Omega_{\varepsilon}$\,,} \\[2mm]
\curl\Curl\epsilon = 0 & \text{in $\Omega_{\varepsilon}$\,,} \\\displaystyle \int_{\partial B_\varepsilon^i} (\epsilon_{rq,c}-\epsilon_{qc,r})\,\ud x_q = \tilde\theta(B_\varepsilon^i) & \text{for $i = 1,\ldots, N$,}\\[3mm]
\displaystyle \int_{\partial B_\varepsilon^i} [\epsilon_{rc}-x_q(\epsilon_{rc,q}-\epsilon_{cq,r})]\,\ud x_c = (\tilde\alpha(B_\varepsilon^i))_r & \text{for $i = 1,\ldots, N$ and $r=1,2$,} 
\end{cases}
\end{equation}
where the boundaries are oriented counter-clockwise,
where $\tilde\alpha\in\ED(\Omega)$ and $\tilde\theta\in\WD(\Omega)$ are defined in \eqref{eq_extended_measures}.
\end{theorem}
\begin{proof}
The proof will rely on three ingredients: (i) a suitable extension to the three-dimensio\-nal setting by considering the cylindrical domain $Y=\Omega_\varepsilon\times(0,1)\subset\R^{3}$, in which we will be able to apply Ting's Theorem~\ref{thm_Ting}, (ii) the structure of the energy, which will be the same for fields $\epsilon\in L^2(\Omega_\varepsilon;\R^{2\times2}_{\sym})$ and for their extensions $\underline{\epsilon}\in L^2(\Omega_\varepsilon;\R^{3\times3}_{\sym})$, and (iii) the properties of $\epsilon^\ii$ (see Proposition~\ref{prop_properties_of_eps_p}).
We break the proof down into four steps.

\noindent \emph{Step 1 (solution to a minimum problem in $K(\Omega_\varepsilon\times(0,1))$).} Let us consider the energy functional
\begin{equation}\label{eq_energy3dTing}
\underline{\mathcal{E}}(E;\Omega_\varepsilon\times(0,1))\coloneqq 
\frac12\int_{\Omega_\varepsilon\times(0,1)} (\C E(x)+\underline{\sigma}^\ii(x)):(E(x)+\underline{\epsilon}^\ii(x))\,\ud x\,;
\end{equation}
by applying the direct method   of the calculus of variations, the minimum problem
\begin{equation}\label{eq_full3d}
\min\{\underline{\E}(E;\Omega_\varepsilon\times(0,1)):E\in K(\Omega_\varepsilon\times(0,1))\}
\end{equation}
has a unique minimizer $\widehat{E}\in K(\Omega_\varepsilon\times(0,1))$.
For every $H\in K(\Omega_\varepsilon\times(0,1))$, we compute
\begin{equation}\label{eq_ELH}
\begin{split}
0=&\, \frac{\ud}{\ud t}\bigg|_{t=0} \underline{\E}(\widehat{E}+tH;\Omega_\varepsilon\times(0,1)) 
= \int_{\Omega_\varepsilon\times(0,1)} \big(\C\widehat{E}+\underline{\sigma}^\ii\big):H\,\ud x \\
=& \int_{\Omega_\varepsilon\times(0,1)} \big(\C\widehat{E}+\underline{\sigma}^\ii\big):\nabla V\,\ud x \\
=&\, \int_{\partial(\Omega_\varepsilon\times(0,1))} \big\langle\big(\C\widehat{E}+\underline{\sigma}^\ii\big) n, V\big\rangle\,\ud\mathcal{H}^2
%\\&\, 
- \int_{\Omega_\varepsilon\times(0,1)} \big\langle\Dive\big(\C\widehat{E}\big), V\big\rangle\,\ud x\,,
\end{split}
\end{equation}
where we have applied Ting's Theorem~\ref{thm_Ting} with $Y=\Omega_\varepsilon\times(0,1)$ to $H$ (obtaining $V\in H^1(\Omega_\varepsilon\times(0,1);\R^{3})$ such that $H=\nabla^{\sym}V$) and used the symmetry of $\C\widehat{E}+\underline{\sigma}^\ii$ to substitute $\nabla^{\sym}V$ with $\nabla V$; to obtain the last line, we have integrated by parts and used that $\Dive\underline{\sigma}^\ii=0$ (by construction of $\sigma^\ii$, see \eqref{eq_Airyoperatorgen}, and by the fact that it is independent of the $x_3$-variable).
This yields the Euler--Lagrange equations
\begin{equation}\label{eq_EL_E_Y}
\begin{cases}
\Dive(\C \widehat{E})=0 & \text{in $\Omega_\varepsilon\times(0,1)$,} \\
\C \widehat{E}\,n=-\C \underline{\epsilon}^\ii\,n & \text{on $(\partial\Omega_\varepsilon)\times(0,1)$,} \\
\C \widehat{E}\,n=\sigma_{33}^\ii e_3 & \text{on $\Omega_\varepsilon\times\{0\}$,}\\
\C \widehat{E}\,n=-\sigma_{33}^\ii e_3 & \text{on $\Omega_\varepsilon\times\{1\}$,}
\end{cases}
\end{equation}
which characterize the unique minimizer $\widehat{E}\in K(\Omega_\varepsilon\times(0,1))$ to \eqref{eq_full3d}.
Notice that in the second equation the outer unit normal $n$ to $(\partial\Omega_\varepsilon)\times(0,1)$ has vanishing third component, so that only the upper left $2\times2$ block $\hat{e}$ of $\widehat{E}$ and $\epsilon^\ii$ contribute to the boundary condition, which can be rewritten as
$$\C \hat{e}(n_1,n_2)^\top=-\C\epsilon^\ii(n_1,n_2)^\top\qquad\text{on $\partial\Omega_\varepsilon$.}$$
Moreover, by Ting's Theorem~\ref{thm_Ting}, there exists a vector field $\widehat{U}\in H^1(\Omega_\varepsilon\times(0,1);\R^{3})$ such that $\widehat{E}=\nabla^{\sym}\widehat{U}$; this implies that 
$$\CURL\CURL\widehat{E}=0\qquad\text{in $H^{-2}(\Omega_\varepsilon\times(0,1);\R^{3\times3}_{\sym})$.}$$

\noindent\emph{Step 2 (solution to the minimum problem \eqref{eq_min_3d}).}
Given $\epsilon\in L^2(\Omega_\varepsilon(D);\R^{2\times2}_{\sym})$,
%and $Y=\Omega_\varepsilon(\defects)\times(0,1)$ (which satisfies the hypotheses of Theorem~\ref{thm_Ting}), 
by applying the direct method of the calculus of variations and by the weak closedness of $K'(\Omega_\varepsilon)$ in $L^2(\Omega_\varepsilon\times(0,1))$, we obtain the existence and uniqueness of a minimizer $\hat{\epsilon}\in K'(\Omega_\varepsilon)$ for the minimum problem \eqref{eq_min_3d}. 
In order to find the necessary conditions for minimality, we first notice that an immediate computation yields the equality, for every $e\in K'(\Omega_\varepsilon)$,
\[\begin{split}
\underline{\E}(\underline{e};\Omega_\varepsilon\times(0,1))=&\, \frac12\int_{\Omega_\varepsilon\times(0,1)} (\C\underline{e}(x_1,x_2)+\underline{\sigma}^\ii(x_1,x_2)):(\underline{e}(x_1,x_2)+\underline{\epsilon}^\ii(x_1,x_2))\,\ud x_1\ud x_2\ud x_3 \\
=&\, \frac12\int_{\Omega_\varepsilon} (\C e(x_1,x_2)+\sigma^\ii(x_1,x_2)):(e(x_1,x_2)+\epsilon^\ii(x_1,x_2))\,\ud x_1\ud x_2 = \E(e),
\end{split}\]
where $\E$ is the energy defined in \eqref{eq_E},
from which we obtain that 
\begin{equation}\label{eq_min_prob_ineq}
\begin{split}
\min\{\E(\epsilon):\epsilon\in K'(\Omega_\varepsilon)\} = &\, \min\{\underline{\E}(\underline{\epsilon};\Omega_\varepsilon\times(0,1)):\epsilon\in K'(\Omega_\varepsilon)\}, \\ %\label{eq_min_prob_same} \\
\geq &\, \min\{\underline{\E}(E;\Omega_\varepsilon\times(0,1)):E\in K(\Omega_\varepsilon\times(0,1))\}; 
\end{split}
\end{equation}
inequality \eqref{eq_min_prob_ineq} is a consequence of the monotonicity of the minimum with respect to the size of competitors.
Then, for every $h\in K'(\Omega_\varepsilon)$, we can compute
\[\begin{split}
0=&\, \frac{\ud}{\ud t}\bigg|_{t=0} \E(\hat{\epsilon}+th)= \frac{\ud}{\ud t}\bigg|_{t=0} \underline{\E}(\hat{\underline{\epsilon}}+t\underline{h};\Omega_\varepsilon\times(0,1)),
\end{split}\]
which, since $\underline{h}\in K(\Omega_\varepsilon\times(0,1))$ is an admissible variation $H$ for \eqref{eq_ELH}, yields that $\hat{\underline{\epsilon}}$ satisfies the Euler--Lagrange equations \eqref{eq_EL_E_Y}.
Therefore, we have proved that the unique minimizer $\hat{\epsilon}$ of the minimum problem \eqref{eq_min_3d} is such that $\hat{\underline{\epsilon}}$ is the unique minimizer \eqref{eq_full3d}, and this makes the inequality in \eqref{eq_min_prob_ineq} an equality.

From now on, we call $\epsilon^\cc$ this minimizer. 
By Ting's Theorem~\ref{thm_Ting}, there exists a vector field $u^\cc\in H^1(\Omega_\varepsilon\times(0,1);\R^{3})$
 %(we prove the structure \eqref{eq_explicit_ue} in Lemma~\ref{lem_special_form})
  such that $\underline{\epsilon}^\cc=\nabla^{\sym} u^\cc$ in $\Omega_\varepsilon\times(0,1)$; it follows that 
$\INC\underline{\epsilon}^\cc=\CURL\CURL\underline{\epsilon}^\cc=0$ in $H^{-2}(\Omega_\varepsilon;\R^{3\times3}_{\sym})$,
which amounts to 
\begin{equation}\label{eq_curlCurl_epsilon^e}
\inc\epsilon^\cc=\curl\Curl\epsilon^\cc=0\qquad\text{in $H^{-2}(\Omega_\varepsilon)$.}
\end{equation}
It is a matter of straightforward calculations to verify that such $u^\cc$ has the structure as in \eqref{eq_explicit_ue}.

\noindent\emph{Step 3 (necessary conditions of the minimality of $\epsilon^\cc$).} 
The first two lines in \eqref{eq_EL_E_Y} and \eqref{eq_curlCurl_epsilon^e} are the first three lines of \eqref{cauchy_ee}.
The existence of $u^\cc$ obtained from Ting's Theorem~\ref{thm_Ting} means that the strain $\epsilon^\cc$ is compatible, so that, by the smoothness of $\partial\Omega$ and by \cite[Proposition 2.8]{Yavari2013} also the conditions on the boundary integrals in \eqref{cauchy_ee} hold true.

\noindent\emph{Step 4 (necessary conditions for $\epsilon=\epsilon^\cc+\epsilon^\ii$).}
The linearity of the differential operators involved and the properties of $\epsilon^\ii$ from Proposition~\ref{prop_properties_of_eps_p} give \eqref{eq_sistema6.7}; the proof is complete.
\end{proof} 

\section{Conclusions}
{We have presented a general variational framework for characterizing the equilibrium of mechanical systems in the presence of kinematic incompatibility constraints, under the assumptions of linearized kinematics in the plane strain regime.
For non-simply connected domains, our theory models the solutions that arise from violations of kinematic compatibility conditions along internal boundaries.
Our study is motivated by the core-radius approach to the study of Volterra wedge disclinations and edge dislocations, where small disks of fixed radius $\varepsilon>0$ around the defects are removed, therefore creating a non-simply connected domain.
%as regularized Volterra-type wedge disclinations and edge dislocations, following the core radius modeling approach.
%As an initial study, we focused on analyzing regularized models with a fixed core radius. 
The parameter $\varepsilon$ characterizing the regularized solutions predicted by this approach 
%are characterized by a fixed length scale $\varepsilon>0$, which 
represents the minimal scale above which the continuum hypothesis remains valid. 
This length scale, typically of the order of nanometers, can be estimated from experimental observations.
%In addition, we characterized the mechanical equilibrium using cell formulas.

We showed that the infinite-dimensional minimization problem can be reduced to a finite-dimensional one via an orthogonal decomposition of the admissible function space. This reduction allowed us to derive a closed-form expression for the Airy stress function in terms of a linear combination of solutions to cell formula problems. An important application of these formulas lies in the numerical simulation of linearized plane strain elasticity. By solving the cell formula problems, which consist of a set of Dirichlet boundary value problems that can be easily implemented in a Finite Element solver, one can reconstruct the Airy stress function and thereby recover the associated strain and stress fields.

Future work will address the rigorous analysis of the regularized system in the asymptotic limit 
$\varepsilon\to 0$, the homogenization of systems of disclinations and dislocations, and the extension of this planar modeling framework to more general geometries such as plates and shells.

%\blu{commentare sulla scelta delle \eqref{eq_v_dD} e dire, sostanzialmente, che isoliamo le sorgenti essenziali di $\Delta^2 v=\de$ e non consideriamo tutto il testo. questo è coerente con il nostro lavoro \cite{CDLM2024}, tuttavia la scelta non è unica, perché: a meno di funzioni affini non altera la PDE; a meno di funzioni biarmoniche non altera l'equazione di campo (ma potrebbe alterare le condizioni al bordo). Se si cambiasse $v^\ii$ con un'altra che ha pezzi nel nucleo di $\Delta^2$, cambierebbe di conseguenza anche la $v^\cc$. Dovremo capire quali impatti ha questa cosa (o semplicemente dire che alcuni campi cambiano e amen).}

\appendix
\section{Some useful results}
We start by stating and proving a simple integration lemma which will be useful in the paper.
For a curve $\gamma\colon[0,L]\to\R^2$ parameterized by arch-length and oriented counter-clockwise, the unit tangent and the unit normal vectors are 
\begin{equation}\label{202308152252}
t(\lambda)=\gamma'(\lambda)=
\!\begin{pmatrix}
x_1'(\lambda)\\
x_2'(\lambda)
\end{pmatrix}\!=\!\begin{pmatrix}
\ud x_1\\
\ud x_2
\end{pmatrix}
\quad\text{and}\quad
n(\lambda)=\Pi t(\lambda)=
\!\begin{pmatrix}
x_2'(\lambda)\\
-x_1'(\lambda)
\end{pmatrix}\!=\!\begin{pmatrix}
\ud x_2\\
-\ud x_1
\end{pmatrix}.
\end{equation}
\begin{lemma}\label{202308142312}
Let $\gamma\colon[0,L]\to\R^2$ be a Jordan curve parameterized by arc length, and let $f\colon\R^2\to\R$ be a (single-valued) function.
Then we have 
\begin{equation}\label{202308142315}
\int_\gamma f\,\ud x_r =  -\int_\gamma x_r\partial_t f\,\ud \Huno \qquad\text{for $r=1,2$.} 
\end{equation}
\end{lemma}
\begin{proof}
Let $\gamma(\lambda)=(x_1(\lambda),x_2(\lambda))$ for every $\lambda\in[0,L]$. 
Equalities \eqref{202308142315} follow from the fact that, thanks to the single-valuedness of the maps $(x_1,x_2)\mapsto x_r\,f(x_1,x_2)$ (for $r=1,2$), we have (recalling that $\partial_t=\ud/\ud\lambda$)
\begin{equation*}
0=\int_\gamma \frac{\ud}{\ud \lambda} (x_rf)\,\ud \lambda=\int_\gamma f\, x_r'\,\ud \lambda +\int_\gamma x_r\frac{\ud f}{\ud\lambda}\,\ud \lambda = \int_\gamma f\,\ud x_r + \int_\gamma x_r \partial_t f\,\ud  \Huno\,. \qedhere
\end{equation*}
\end{proof}

We prove here a technical lemma about integration by parts.
\begin{lemma}\label{lemma_ibp}
Let $U\subset\R^{2}$ be an open set and let $v\in\mathcal{C}^4(U)$. 
If $\phi\colon U\to\R{}$ is a regular enough function, we have
\begin{equation}\label{eq_ibp}
\begin{split}
&\, \frac{1+\nu}{E}\int_{U} (\nabla^2v:\nabla^2\phi-\nu\Delta v\Delta \phi)\,\ud x\\
=&\, \frac{1-\nu^2}{E}\bigg[\int_{U} \phi\Delta^2 v\,\ud x+\int_{\partial U} \bigg(\Delta v\, \partial_n\phi-\phi\partial_n(\Delta v)+\frac{\langle \nabla^2v\,t,\Pi(\nabla \phi)\rangle}{1-\nu}\bigg)\,\ud\Huno\bigg].
\end{split}
\end{equation}
\end{lemma}
\begin{proof}
The proof is a matter of a simple computation, using the Gauss--Green formula a few times. 
Recalling that 
\[\begin{split}
\nabla^2 v:\nabla^2\phi = &\, v_{xx}\phi_{xx}+v_{xy}\phi_{xy}+v_{yx}\phi_{yx}+v_{yy}\phi_{yy}\,,\\
\Delta v\Delta\phi= &\, v_{xx}\phi_{xx}+v_{xx}\phi_{yy}+v_{yy}\phi_{xx}+v_{yy}\phi_{yy}\,,
\end{split}\]
we can integrate by parts each of the terms above and obtain
\begin{subequations}
\begin{eqnarray}
\int_{U} v_{xx}\phi_{xx}\,\ud x=&\displaystyle\int_{\partial U} (v_{xx}\phi_xn_x-v_{xxx}n_x\phi)\,\ud\Huno+\int_U v_{xxxx}\phi\,\ud x, \label{eq_ibp1}\\
\int_{U} v_{yy}\phi_{yy}\,\ud x=&\displaystyle\int_{\partial U} (v_{yy}\phi_yn_y-v_{yyy}n_y\phi)\,\ud\Huno+\int_U v_{yyyy}\phi\,\ud x, \label{eq_ibp2}\\
\int_{U} v_{xy}\phi_{xy}\,\ud x=&\displaystyle\int_{\partial U} (v_{xy}\phi_xn_y-v_{xyy}n_x\phi)\,\ud\Huno+\int_U v_{xxyy}\phi\,\ud x, \label{eq_ibp3}\\
\int_{U} v_{yx}\phi_{yx}\,\ud x=&\displaystyle\int_{\partial U} (v_{xy}\phi_yn_x-v_{xxy}n_y\phi)\,\ud\Huno+\int_U v_{xxyy}\phi\,\ud x, \label{eq_ibp4}\\
\int_{U} v_{xx}\phi_{yy}\,\ud x=&\displaystyle\int_{\partial U} (v_{xx}\phi_yn_y-v_{xxy}n_y\phi)\,\ud\Huno+\int_U v_{xxyy}\phi\,\ud x, \label{eq_ibp5}\\
\int_{U} v_{yy}\phi_{xx}\,\ud x=&\displaystyle\int_{\partial U} (v_{yy}\phi_xn_x-v_{yyx}n_x\phi)\,\ud\Huno+\int_U v_{xxyy}\phi\,\ud x. \label{eq_ibp6}
\end{eqnarray}
\end{subequations}
By using \eqref{eq_ibp1}--\eqref{eq_ibp4} and by adding and subtracting $v_{yy}\phi_xn_x$ and $v_{xx}\phi_yn_y$\,, we have
\[\begin{split}
\int_{U} \nabla^2v:\nabla^2\phi\,\ud x=& \int_{U} (v_{xx}\phi_{xx}+v_{xy}\phi_{xy}+v_{yx}\phi_{yx}+v_{yy}\phi_{yy})\,\ud x\\
=& \int_U \Delta^2v\,\phi\,\ud x+\int_{\partial U} \big(\Delta v\,\partial_n\phi - \phi\partial_n(\Delta v)\big)\,\ud\Huno\\
&\,+\int_{\partial U} (v_{xy}n_y\phi_x+v_{xy}n_x\phi_y-v_{xx}n_y\phi_y-v_{yy}n_x\phi_x)\,\ud\Huno\\
=& \int_U \Delta^2v\,\phi\,\ud x+\int_{\partial U} \big(\Delta v\,\partial_n\phi - \phi\partial_n(\Delta v)+\langle \nabla^2v\,t,\Pi(\nabla\phi)\rangle \big)\, \ud\Huno;\\
\end{split}\]
by using \eqref{eq_ibp1}, \eqref{eq_ibp2}, \eqref{eq_ibp5}, and \eqref{eq_ibp6}, and by adding and subtracting $v_{yy}\phi_xn_x$ and $v_{xx}\phi_yn_y$\,, we have
\[\begin{split}
\int_{U} \Delta v\Delta\phi\,\ud x=& \int_U (v_{xx}\phi_{xx}+v_{xx}\phi_{yy}+v_{yy}\phi_{xx}+v_{yy}\phi_{yy})\,\ud x \\
=& \int_{U} \Delta^2v\,\phi\,\ud x+\int_{\partial U} \big(\Delta v\,\partial_n\phi-\phi\partial_n(\Delta v)\big)\,\ud\Huno.
\end{split}\]
Formula \eqref{eq_ibp} follows.
\end{proof}

\section{Proof of Proposition \ref{prop_boundary_cond_incomp}}\label{app_B}
Before proving the proposition, we present a technical lemma which contains the essential details of the proof of \eqref{eq_necessary_incompatible_boundary}. 
We will consider a sample domain with the topology of an annulus.
We will suppose, without loss of generality, that $\Omega=B_1(0)$, that $\rho<1$, and that $\defects=\{0\}$, so that $\Omega^{\nsc}=B_1(0)\setminus\overline{B}_\rho(0)$\,; finally, we let $\Gamma\coloneqq\partial B_\rho(0)$.
\begin{lemma}\label{lemma_technical_defects}
Let $A\subset\R^2$ be as above. 
For $s\in\R\setminus\{0\}$ and $b\in\R^2\setminus\{0\}$, let $\theta=s\de_0\in\WD(\Omega)$ and $\alpha=b\de_0\in\ED(\Omega)$; moreover, let $v_d^s\coloneqq -sv_d \colon\R^2\to\R$ and $v_D^b\coloneqq -b\times v_D \colon\R^2\to\R$, respectively, where $v_d$ and $v_D$ are defined in \eqref{eq_v_dD}, that is
\begin{equation}\label{eq_vdvDbar}
v_d^s(x)=-\frac{sE}{1-\nu^2}\frac{|x|^2}{16\pi}\log|x|^2
\qquad\text{and}\qquad
v_D^b(x)=-\frac{E}{1-\nu^2}\frac{b\times x}{8\pi}(\log|x|^2+1).
\end{equation}
Then we have 
\begin{equation}\label{eq_202409051843}
\begin{cases}
\displaystyle \frac{1-\nu^2}{E}\int_{\Gamma} \partial_n(\Delta v_d^s)\,\ud\Huno=s=\theta(\Omega)\,,\\[3mm]
\displaystyle \frac{1-\nu^2}{E}\int_{\Gamma} \bigg( x_1\partial_t(\Delta v_d^s)-x_2\partial_n(\Delta v_d^s) +\frac{(\nabla^2 v_d^s\, t)_1}{1-\nu}\bigg)\,\ud\Huno=0\,,\\[3mm]
\displaystyle \frac{1-\nu^2}{E}\int_{\Gamma} \bigg( x_1\partial_n(\Delta v_d^s)+x_2\partial_t(\Delta v_d^s) +\frac{(\nabla^2 v_d^s\, t)_2}{1-\nu}\bigg)\,\ud\Huno=0\,,
\end{cases}
\end{equation}
and
\begin{equation}\label{eq_2024090518431}
\begin{cases}
\displaystyle \frac{1-\nu^2}{E} \int_{\Gamma} \partial_n(\Delta  v_D^b)\,\ud\Huno=0\,,\\[3mm]
\displaystyle \frac{1-\nu^2}{E}\int_{\Gamma} \bigg( x_1\partial_t(\Delta v_D^b)-x_2\partial_n(\Delta v_D^b) +\frac{(\nabla^2 v_D^b\, t)_1}{1-\nu}\bigg)\,\ud\Huno=b_1=(\alpha(\Omega))_1\,,\\[3mm]
\displaystyle \frac{1-\nu^2}{E}\int_{\Gamma} \bigg( x_1\partial_n(\Delta v_D^b)+x_2\partial_t(\Delta v_D^b) +\frac{(\nabla^2 v_D^b\, t)_2}{1-\nu}\bigg)\,\ud\Huno=b_2=(\alpha(\Omega))_2\,,
\end{cases}
\end{equation}
where $n$ is the outer unit normal to $\partial\Omega^\nsc$ and $t$ points clockwise.
\end{lemma}
\begin{proof}
We start by noting that the function $v_d^s$ and $v_D^s$ defined in the statement of the lemma solve
\begin{equation}\label{eq_202409051736}
\frac{1-\nu^2}{E}\Delta^2 v_d^s=-s\de_0
\qquad\text{and}\qquad
\frac{1-\nu^2}{E}\Delta^2 v_D^b=-(b\times\nabla)\de_0=b_2\partial_{x_1}\de_0-b_1\partial_{x_2}\de_0\,,
\end{equation}
in $H^{-2}(B_1(0))$ and in $\mathcal{D}'(B_1(0))$, respectively.
In order to draw information on the contribution on $\Gamma$, the internal boundary of the annular domain $\Omega^{\nsc}$, we can test the equations in \eqref{eq_202409051736} against functions in $\Cnew(\Omega(\{0\}))$, namely we consider test functions $\varphi\in H^2_0(B_1(0))$ such that $\varphi=a$ in $B_\rho(0)$, for a certain affine function $a$.
These test functions are such that 
$$
\begin{array}{rl}
\varphi|_{\Gamma}=a|_{\Gamma}= & (a_0+a_1x_1+a_2x_2)|_{\Gamma}\\[1mm]
= & a_0+ a_1\rho\cos\frac\lambda\rho+ a_2\rho\sin\frac\lambda\rho
\end{array}
\quad\text{and}\quad
\begin{array}{rl}
\partial_n \varphi|_{\Gamma}=\partial_n a|_{\Gamma}= & -a_1n_{\rho,1}-a_2n_{\rho,2}\\[1mm]
=& -a_1\cos\frac\lambda\rho - a_2\sin\frac\lambda\rho 
\end{array}$$
for $\lambda\in[0,2\pi\rho]$, where $n_\rho=(n_{\rho,1},n_{\rho,2})$ is the outer unit normal to $\partial B_\rho(0)$.

For $v=v_d^s,v_D^b$ and for any $\varphi\in\Cnew(\Omega_\varepsilon(\{0\}))$, up to the factor $(1-\nu^2)/E$, we have, by integration by parts (here we must choose the normal to $\Gamma$ pointing outwards of $\Omega^{\nsc}$, that is towards the origin, that is $-n_\rho$),
\begin{equation}\label{eq_202409051842}
\begin{split}
\int_\Omega \Delta^2 v\,\varphi\,\ud x= & \int_{\Omega^{\nsc}} \Delta^2 v\,\varphi\,\ud x+ \int_{B_\rho(0)} \Delta^2 v\,\varphi\,\ud x\\
= & \int_{\Gamma} \partial_{n_\rho}(\Delta v)\varphi\,\ud\Huno-\int_{B_\rho(0)} \langle \nabla(\Delta v),\nabla \varphi\rangle\,\ud x \\
= & \int_{\Gamma} \partial_{n_\rho}(\Delta v)\varphi\,\ud\Huno-\int_{\Gamma} \Delta v\partial_{n_\rho}\varphi\,\ud\Huno \\ 
= & \int_{\Gamma} \partial_{n_\rho}(\Delta v)[a_0+a_1x_1+a_2x_2]\,\ud\Huno - \int_{\Gamma} \Delta v[a_1n_{\rho,1}+a_2n_{\rho,2}]\,\ud\Huno \\
= & \int_{\Gamma} \partial_{n_\rho}(\Delta v)[a_0+a_1x_1+a_2x_2]\,\ud\Huno + \int_{\Gamma} \partial_{t_\rho}(\Delta v) (a_1x_2- a_2x_1)\,\ud\Huno,
\end{split}
\end{equation}
since $\Delta^2 v$ vanishes in $\Omega^{\nsc}$ and since (being $\varphi$ affine in $B_\rho(0)$) $\Delta \varphi=0$; also, we have used \eqref{202308152252} and \eqref{202308142315} with $f=\Delta v$ in the last equality.
In the case of a disclination, $v= v_d^s$\,, by choosing alternatively only one coefficient of $\varphi$ not to vanish, from the first equation in \eqref{eq_202409051736} and from \eqref{eq_202409051842} we obtain that 
\begin{enumerate}
\item if $a_0\neq0=a_1=a_2$, then $\displaystyle -sa_0=-s\varphi(0)=\int_{\Gamma} a_0\partial_{n_\rho}(\Delta v_d^s)\,\ud\Huno\,$;
\item if $a_1\neq0=a_0=a_2$, then $\displaystyle 0=-s\varphi(0)=\int_{\Gamma} [x_1\partial_{\rho,n}(\Delta v_d^s)+x_2\partial_t(\Delta v_d^s)]\,\ud\Huno\,$;
\item if $a_2\neq0=a_0=a_1$, then $\displaystyle 0=s\varphi(0)=\int_{\Gamma} [x_2\partial_{n_\rho}(\Delta v_d^s)-x_1\partial_{t_\rho}(\Delta v_d^s)]\,\ud\Huno\,$.
\end{enumerate}
Moreover, a straightforward computation yields
$$\nabla^2 v_d^s\,t_\rho|_{\Gamma}=-\frac{sE}{16\pi(1-\nu^2)}
\begin{pmatrix}
2x_2(1+\log|x|^2)\\
-2x_1(1+\log|x|^2)
\end{pmatrix},$$
so that
\begin{equation}\label{eq_Hdt}
\int_{\Gamma} (\nabla^2 v_d^s\,t_\rho)_r\, \ud\Huno=0,\quad\text{for $r=1,2$.}
\end{equation}
The three relations above and \eqref{eq_Hdt} (reinstating the factor $(1-\nu^2)/E$ and recalling that $n=-n_\rho$ and $t=-t_\rho$) yield the necessary conditions \eqref{eq_202409051843}.

Let us now turn to the case of the dislocation, for which $v= v_D^b$\,. 
By testing the right-hand side of the second equation in \eqref{eq_202409051736} with a test function $\varphi$ which is affine on $B_\rho(0)$, we obtain
$$\langle -(b\times\nabla)\de_0 ,\varphi\rangle = \langle b_2\partial_{x_1}\de_0-b_1\partial_{x_2}\de_0, \varphi \rangle= b_1\partial_{x_2}\varphi(0)-b_2\partial_{x_1}\varphi(0)=b_1a_2-b_2a_1\,.$$
By choosing alternatively only one coefficient of $\varphi$ not to vanish, from \eqref{eq_202409051842}  and the chain of equalities above, we obtain that
\begin{enumerate}
\item if $a_0\neq0=a_1=a_2$, then $\displaystyle 0=\int_{\Gamma} \partial_{n_\rho}(\Delta v_D^b)\,\ud\Huno\,$;
\item if $a_1\neq0=a_0=a_2$, then $\displaystyle -b_2=\int_{\Gamma} [x_1\partial_{n_\rho}(\Delta  v_D^b)+x_2\partial_{t_\rho}(\Delta v_D^b)]\,\ud\Huno\,$;
\item if $a_2\neq0=a_0=a_1$, then $\displaystyle -b_1= \int_{\Gamma} [x_2\partial_{n_\rho}(\Delta  v_D^b)-x_1\partial_{t_\rho}(\Delta v_D^b)]\,\ud\Huno\,$.
\end{enumerate}
Moreover, a straightforward computation yields
$$\nabla^2 v_D^b\,t_{\rho}|_{\Gamma}=-\frac{E}{1-\nu^2}\begin{pmatrix}
\displaystyle\frac{b_1(x_2^2-x_1^2)-2b_2x_1x_2}{x^2+y^2}\\
\displaystyle\frac{b_2(x_1^2-x_2^2)+2b_1x_1x_2}{x^2+y^2}
\end{pmatrix},$$
so that 
\begin{equation}\label{eq_HDt}
\int_{\Gamma} (\nabla^2 v_D^b\,t_\rho)_r\, \ud\Huno=0,\quad\text{for $r=1,2$.}
\end{equation}
The three relations above and \eqref{eq_HDt} (reinstating the factor $(1-\nu^2)/E$ and recalling that $n=-n_\rho$ and $t=-t_\rho$) yield the necessary conditions \eqref{eq_2024090518431}.
\end{proof}
\begin{proof}[Proof of Proposition~\ref{prop_boundary_cond_incomp}]
From the discussion in Section~\ref{sec_prel}, in particular from \eqref{eq_incexplicit}, \eqref{eq_formaltrans}, and \eqref{eq_202409041452}, we know that the function $v^\ii$ verifies \eqref{eq_bilaplacianappendix}.
Moreover, \eqref{eq_necessary_incompatible_bulk} follows immediately from \eqref{eq_inc} and \eqref{eq_formaltrans}, since $\spt(\alpha)\cup\spt(\theta)=\defects\cap\Omega_\varepsilon(\defects)=\emptyset$ by construction of $\Omega_\varepsilon(\defects)$.

To prove \eqref{eq_necessary_incompatible_boundary}, we notice that, by linearity, it suffices to look at one defect site $\xi\in\defects$ alone, where we place once a disclination of Frank angle $s$ and once a dislocation of Burgers vector $b$. 
The proof in the prototypical situation $\xi=0$ is dealt with in Lemma \ref{lemma_technical_defects}, and the result in the general case can be obtained by superposition and translations, as in \eqref{eq_vp}.
The only thing that needs to be verified is that 
\begin{equation}\label{eq_cross_terms}
\begin{cases}
\displaystyle \frac{1-\nu^2}{E} \int_{\Gamma^{i_1}} \partial_n(\Delta v^{i_2})\,\ud\Huno=0,\\[3mm] %-\hat\theta(\Omega^{i_1})\,,\\[3mm]
\displaystyle \frac{1-\nu^2}{E} \int_{\Gamma^{i_1}} \bigg(x_1\partial_t(\Delta v^{i_2})-x_2\partial_n(\Delta v^{i_2})+\frac{(\nabla^2 v^{i_2}\,t)_1}{1-\nu}\bigg)\,\ud\Huno=0, %-(\hat\alpha(\Omega^{i_1}))_1\,.
\\[3mm]
\displaystyle \frac{1-\nu^2}{E} \int_{\Gamma^{i_1}} \bigg(x_1\partial_n(\Delta v^{i_2})+x_2\partial_t(\Delta v^{i_2})+\frac{(\nabla^2 v^{i_2}\,t)_2}{1-\nu}\bigg)\,\ud\Huno=0,\\ %-(\hat\alpha(\Omega^{i_1}))_2\,,
\end{cases}
\end{equation}
whenever $i_1\neq i_2$, and where we denote by $v^i$ either one of the functions $v_D^j$ or $v_d^k$ defined in \eqref{eq_vp}, but this is an immediate consequence of the compatibility of $v^{i_2}$ away from~$\xi^{i_1}$.
The proof is concluded.
\end{proof}

\section{Useful properties of the Monge-Amp\`{e}re operator}
We recall here the definition of the Monge-Amp\`{e}re operator and present two properties of interest.
\begin{definition}[Monge-Amp\`{e}re operator]\label{sec:24241126922}
Let $\Omega\subset\mathbb{R}^2$ be an open set; for any $\xi,\eta \in H^2(\Omega)$, the Monge-Amp\`{e}re operator is defined as
\begin{equation}
[\xi, \eta](x)\coloneqq \cof(\nabla^2 \xi(x)) : \nabla^2 \eta(x) \qquad \text{for a.e.~$x \in \Omega$.}
\end{equation}
\end{definition}
\begin{lemma}[{\cite[Theorem~5.8-2]{ciarlet97}}] 
\label{Lemma: Ciarlet's lemma}
Let $\Omega\subset\mathbb{R}^2$ be a bounded, connected, open set with Lipschitz boundary. 
Let $\xi,\eta,\chi\in H^2(\Omega)$ and assume that at least one of them belongs to $H^2_0(\Omega)$. 
Then $\int_{\Omega} [\xi, \eta](x) \chi(x)\, \ud x$ is symmetric, that is
\begin{equation}\label{eq_sympropCiarlet}
\int_{\Omega} [\xi, \eta](x) \chi(x)\,\ud x = \int_{\Omega} [\chi, \xi](x) \eta(x)\,\ud x = 
\int_{\Omega} [\eta, \chi](x) \xi(x)\,\ud x.
\end{equation} 
\end{lemma}
We prove that a slightly weaker version of symmetry property \eqref{eq_sympropCiarlet} holds when one gives up the request that one of the functions be in $H^2_0(\Omega)$, provided that it is in $H^4(\Omega)$ and has affine trace on the boundary $\partial\Omega$.
\begin{lemma}\label{Lemma:2410171413}
Let $\Omega \subset \mathbb{R}^2$ be a bounded, connected, open set with $\Ccal^2$ boundary $\partial\Omega$. 
Let $\eta,\chi\in H^2(\Omega)$ and let $\xi \in H^{4}(\Omega)$ be such that $\xi$ is affine on $\partial \Omega$. %, with possibly different affine constants on the different connected components of $\partial \Omega$.
Then %the following symmetry property holds
\begin{equation}\label{eq_C3}
\int_{\Omega} [\xi, \eta](x)\chi(x)\,\ud x= 
\int_{\Omega} [\xi, \chi](x)\eta(x)\,\ud x.
\end{equation}
\end{lemma}
\begin{proof} 
By definition of the Monge-Amp\`{e}re operator,
\begin{equation}\label{eq_MAsym}
\begin{split}
\int_{\Omega} [\xi, \eta]\chi\,\ud x=& \int_{\Omega} \chi\, \cof(\nabla^2 \xi) : \nabla^2 \eta\,\ud x 
=\int_{\partial \Omega} \chi \langle\cof(\nabla^2 \xi)n, \nabla \eta\rangle \,\ud\Huno \\
&\, - \int_{\Omega} \big\langle \Div\big(\chi \cof(\nabla^2 \xi)\big), \nabla \eta\big\rangle \,\ud x 
=-\int_{\Omega} \langle \cof(\nabla^2\xi) \nabla\chi,\nabla \eta\rangle\,\ud x,
\end{split}
\end{equation}
where the second line follows from integration by parts, and the third line is a consequence of \cite[Proposition~A.2]{CDLM2024} applied to $\xi$ (upon noticing that $\cof(M)=\Pi^\top M\Pi$ for any $M\in\mathbb{R}^{2\times2}_{\sym}$ and that $\Pi n=-t$) and \eqref{eq_Div_sigma}.
Realizing that the right-hand side of \eqref{eq_MAsym} is a symmetric expression in $\eta$ and $\chi$, \eqref{eq_C3} follows.
\end{proof}

\enlargethispage{20pt}

\noindent\textbf{Acknowledgments.} PC's work is supported by JSPS KAKENHI
Grant-in-Aid for Scientific Research (C) JP24K06797.
PC holds an honorary appointment at La Trobe University.
EF is supported by JST SPRING, Grant Number JPMJSP2136.
MM warmly thanks   the Institute of Mathematics for Industry, an International Joint Usage and Research Center located in Kyushu University, where part of the work contained in this paper was carried out.
PC and MM are members of the \emph{Gruppo Nazionale per l'Analisi Matematica, la Probabilità e le loro Applicazioni (GNAMPA)} of the \emph{Istituto Nazionale di Alta Matematica (INdAM)}. MM acknowledges partial support from the MUR grant \emph{Geometric Analytic Methods for PDEs and Applications} (2022SLTHCE cup E53D23005880006). This manuscript reflects only the authors’ views and opinions and the Italian Ministry cannot be considered responsible for them.

\noindent\textbf{Author Contributions.} All authors developed the model, wrote the main manuscript text, and reviewed the manuscript.

\noindent\textbf{Data Availability.} No datasets were generated or analyzed during the current study.

\noindent\textbf{Competing Interests.} The authors declare no competing interests.

\vskip2pc

\bibliographystyle{siam}
\bibliography{refsPatrick}

%\noindent {\bf Please follow the coding for references as shown below.}
%
%\begin{thebibliography}{9}
%
%\bibitem{1} Allwood JM, Cullen JM. 2011 \textit{Sustainable materials:  with both eyes open}.
%Cambridge, UK: UIT Cambridge. See \href{http://www.withbotheyesopen.com}{http://www.withbotheyesopen.com}.
%
%\bibitem{2}  MacKay DJC. 2008  \textit{Sustainable energy:  without the hot air}.
% Cambridge, UK: UIT Cambridge. See \href{http://www.withouthotair.com}{http://www.withouthotair.com}.
%
%\bibitem{3} Gallman PG. 2011  \textit{Green alternatives and national energy strategy: the facts
% behind the headlines}.  Baltimore,\ MD: Johns Hopkins University Press.
%
%\bibitem{4} MacKay DJC. 2013.  Solar energy in the context of energy use, energy transportation, and
% energy storage. \textit{Proc. R. Soc. A} \textbf{371}.
%
%\end{thebibliography}
%
%\noindent If maintaining .bib file for references, then please use "RS.bst" to generate the references.
%
%\noindent Example:
%
%\verb+\bibliographystyle{RS}+ %%%% .BST file
%
%\verb+\bibliography{sample}+ %%%%% .Bib file

\end{document}